\documentclass[preprint,12pt]{elsarticle}

\usepackage{amssymb}
\usepackage{amsmath}
\usepackage{amsthm}

\usepackage[linesnumbered,algoruled,noline,longend]{algorithm2e}

\journal{Journal of Computational Algebra}

\newtheorem{theorem}{Theorem}[section]

\newtheorem{corollary}[theorem]{Corollary}
\newtheorem{proposition}[theorem]{Proposition}

\theoremstyle{definition}
\newtheorem{definition}[theorem]{Definition}
\newtheorem{example}[theorem]{Example}
\newtheorem{remark}[theorem]{Remark}

\newcommand{\PSLR}{\mathop{\rm PSL}(2,\mathbb{R})}
\newcommand{\SLR}{\mathop{\rm SL}(2,\mathbb{R})}
\newcommand{\tr}{\mathop{\rm tr}\nolimits}
\newcommand{\ord}{\mathop{\rm ord}\nolimits}
\newcommand{\simN}{\underset{N}{\sim}}
\newcommand{\ZZ}{{\mathbb{Z}}}

\let\phi\varphi

\begin{document}

\begin{frontmatter}

\title{Trace Minimization and Roots in ${\rm PSL}(2,\mathbb{R})$}


\author[1]{Martin Kreuzer\corref{cor1}}
\ead{Martin.Kreuzer@uni-passau.de}

\author[2]{Anja Moldenhauer}
\ead{A.Moldenhauer@gmx.net}

\author[3]{Gerhard Rosenberger}%
\ead{Gerhard.Rosenberger@icloud.com}

\affiliation[1]{organization={Faculty of Informatics and Mathematics},
            addressline={University of Passau}, 
            city={94030 Passau},
            country={Germany}}

\affiliation[2]{
            addressline={Hammer Steindamm 111}, 
            city={20535 Hamburg}, 
            country ={Germany}}

\affiliation[3]{organisation={University of Hamburg},
            addressline={Bundesstr. 55},
            city={20146 Hamburg},
            country={Germany}}

\cortext[cor1]{Corresponding author}

\begin{abstract}
Suppose that $A,B \in {\rm PSL}(2,\mathbb{R})$ generate a non-elementary Fuchsian group.
Let $m,n\in\mathbb{N}_+$, and let $R,S\in {\rm PSL}(2,\mathbb{R})$ such that
$R^m=A$ and $S^n=B$. We present explicit algorithms to check whether $\langle R,S\rangle$
is a Fuchsian group. These algorithms rely only on the knowledge of the traces
$\tr(A)$, $\tr(B)$, and $\tr(AB)$, which we assume to be given as algebraic numbers.
The main tools are the classic Trace Minimization Algorithm, as introduced in 1972 
by the third author, a new Extended Trace Minimization Algorithm, and a
Rational Angle Recovery Algorithm which checks whether a given number $x$
is if the form $x = 2 \,\cos(p\,\pi/q)$.

The questions when roots of the generators of a free Fuchsian group of rank~2
generate again a free Fuchsian group of rank~2, and an extension to
positive rational exponents $m,n$ are treated, as well.
\end{abstract}




\end{frontmatter}

%
%

\section{Introduction}
\label{sec1}

The special linear group $\PSLR$ can be identified with the group of orientation-preserving
isometries of the hyperbolic plane. Since discrete subgroups of $\PSLR$ correspond to
tesselations of the hyperbolic plane through their fundamental domains, the classification
of these discrete subgroups has been an important topic in group theory. 
They are sometimes called \textit{Fuchsian groups}. 
In this paper we consider only non-elementary Fuchsian groups, i.e., we
exclude cyclic groups and groups isomorphic to the infinite dihedral group. 

A particularly successful story is the classification of 2-generated Fuchsian groups.
Starting with the results of A.W.\ Knapp~\cite{Kna}, who considered the cases when neither
of the two generators is hyperbolic, this classification was completed in the 1970s
through the works of N.\ Purzitsky and the third author \cite{Pur1,Pur2,PR,Ros0,Ros1},
and later laid out comprehensively in~\cite{Ros2,FR}. 

For an algorithmic treatment of this classification, the groundwork was done in the 
third author's doctoral thesis~\cite{Ros0}
where he introduced the trace minimization technique together with a finiteness proof.
This technique was later used in numerous papers (see e.g.~\cite{FR,KIR,Ros3,Ros4}), 
but it was usually referred to in an informal way, except for a version to test freeness in~\cite{EKL}.
One of the results of this paper is that we present an explicit version
of this approach which can be implemented in any compute algebra system that supports algebraic numbers.

A different path to test whether a 2-generated subgroup of $\PSLR$ is discrete is based
on geometric arguments and calculations. It was initiated by J.G.\ Matelski~\cite{Mat}, expanded and
improved by J.\ Gilman and B.\ Maskit in~\cite{GM}, and finally completed by J.\ Gilman in~\cite{G3}.
The computations underlying this approach are mainly geometrric and have as input the knowledge of the entries 
of the matrices $A,B$ in the form of algebraic numbers (see~\cite[Theorem 14.5.1]{G3}).

It follows from results of Majeed~\cite{Maj}, Rosenberger~\cite{Ros3}, as well as Fine and Rosenberger~\cite{FR}
that, given a 2-generated subgroup $G=\langle A,B\rangle$ of~$\PSLR$ which is 
non-elementary, one can pass from~$(A,B)$ to a Nielsen equivalent
pair of generators $(U,V)$ of~$G$ such that the subgroup $\langle U^m, V^n\rangle$ is a free
Fuchsian group of rank~2 for sufficiently large $m,n\gg 0$. 
(For a detailed proof, see~\cite[Thm.~1.7.50 and Cor.~1.7.52]{FMRSW}.)

In this note we are interested in the reverse question: under which conditions do roots
of generators of 2-generated Fuchsian groups generate a Fuchsian group again? 
Recall that a 2-generated Fuchsian group is either free of rank~2, 
or it contains elliptic elements of finite order (cf.\ \cite{Ros1}). 
In the following we assume that $\overline{A}, \overline{B}\in \PSLR$ are elements which 
generate a non-elementary Fuchsian group 
$G = \langle \overline{A}, \overline{B}\rangle$, and that
$\overline{R}, \overline{S} \in \PSLR$ are such that $\overline{R}^{\,m} = \overline{A}$ and 
$\overline{S}^{\,n} = \overline{B}$ for some $m,n\in \mathbb{N}_+$.
To simplify the presentation, we shall always choose representatives $A,B,R,S \in \SLR$
whose traces are non-negative and work with these matrices. 
In this setting we address the following two questions.
\begin{enumerate}
\item[(1)] Given the trace triple $(\tr(A), \tr(B), \tr(AB))$, can we decide
whether the larger group $\langle R,S\rangle$ is discrete?

\item[(2)] Given the trace triple $(\tr(A), \tr(B), \tr(AB))$ of a free Fuchsian group of rank~2, 
can we decide whether the larger group $\langle R,S\rangle$ is a free Fuchsian
group of rank~2?
\end{enumerate}

In several papers, special cases of these questions have been treated previously
(see~\cite{Bea2,G1,G3,Par}), but a general answer seems to be unknown.
The solutions given here are partially algorithmic, i.e., sometimes one has to apply an algorithm
to decide the question. The input of these algorithms is the
trace triple $(\tr(A),\tr(B),\tr(AB))$ which we assume to be given as algebraic numbers
via their minimal polynomials and isolating intervals. Note that, if some of these traces are
transcendental numbers, the classification theorems below remain valid, but an actual
implementation may be impossible.

The {\it commutator trace} $\tr([A,B])$ is computed via the Fricke identity
$$
\tr([A,B]) \;=\; \tr(A)^2 + \tr(B)^2 +\tr(AB)^2 - \tr(A)\tr(B)\tr(AB) - 2.
$$
Notice that it does not depend on the choice of the representatives $A,B$, and
it satisfies $\tr([A,B])\ne 2$, since we assumed the Fuchsian group $\langle A,B\rangle$
to be non-elementary. 

The first step in our algorithms will usually be to compute
the trace triple and the commutator trace of the pair $(R,S)$.
For this we use the \textit{Chebychev S-polynomials} $S_i(x)$ defined by $S_0(x)=0$,
$S_1(x)=1$, and $S_i(x) = x\, S_{i-1}(x) - S_{i-2}(x)$ for $i\ge 2$.
The identity element of $\PSLR$ will be denoted by~$E$.
The following result is not new, but crucial for our approach.

\begin{theorem}[Computing the Trace Triple of $(R,S)$]\label{thm:CompTraceTriple}\, \\
Let $A,B \in \PSLR \setminus \{E\}$, let $x=\tr(A)\ge 0$, let $y=\tr(B)\ge 0$, 
let $z=\tr(AB)$, and let $m,n\in \mathbb{N}_+$.
\begin{enumerate}
\item[(a)] If $x=2$, i.e., if~$A$ is parabolic, then $\tr(R)=2$. Similarly, if $y=2$,
then $\tr(S)=2$.

\item[(b)] If $x>2$, i.e., if $A$ is hyperbolic, then
$$
\tr(R) \;=\; \sqrt[m]{\tfrac{1}{2}\,(x + \sqrt{x^2 - 4})} + \sqrt[m]{\tfrac{1}{2}\, (x - \sqrt{x^2 - 4})}
$$
The analogous formula holds for $\tr(S)$ if $\tr(B)>2$.

\item[(c)] If $0\le x<2$, i.e., if~$A$ is elliptic, then there are~$m$ 
possible values of~$\tr(R)$, namely 
$$
\tr(R) = \sqrt[m]{\tfrac{(-1)^\ell}{2} (x + \sqrt{x^2 - 4})} + \sqrt[m]{\tfrac{(-1)^\ell}{2} (x - \sqrt{x^2 - 4})}
$$
where $\ell \in \{0,\dots,m-1\}$.
The analogous formula holds for $\tr(S)$ if $0\le \tr(B) <2$.

\item[(d)] If $x$ is an algebraic number with minimal polynomial $P(x)$ then the minimal polynomial
of $x'=\tr(R)$ is one of the irreducible factors of the polynomial $P(x'\, S_m(x') - 2\, S_{m-1}(x'))$.  
The analogous statement holds for $\tr(S)$ if~$y$ is an algebraic number.

\item[(e)] Let $x'=\tr(R)$ and $y'=\tr(S)$. If $m>1$ and $n>1$, we have $S_m(x')\ne 0$ 
as well as $S_n(y')\ne 0$, and $\tr(RS)$ is given by
$$
\tr(RS) \;=\; \tfrac{1}{S_m(x') S_n(y')}\, (z + S_{m+1}(x') S_{n-1}(y') + S_{m-1}(x') S_{n+1}(y')).
$$
If $m=1$ and $n>1$ then $\tr(RS)= \frac{1}{S_n(y')}\, (z + x\, S_{n-1}(y')$. The analogous formula holds for
$m>1$ and $n=1$.
\end{enumerate}
\end{theorem}

Thus the trace triple of $(R,S)$ consists of algebraic numbers whose minimal polynomials
and isolating intervals can be computed. Using the Fricke identity, we can then compute
the algebraic number $\tr([R,S])$.

Our answers to the above two questions are based on three main algorithms: the first one is the
classical Trace Minimization Algorithm which was first introduced in 1972 in~\cite{Ros0}. For, an explicit formulation 
see Algorithm~\ref{alg:TraceMin}. The second algorithm is called the Extended Trace Minimization
Algorithm (see Algorithm~\ref{alg:ExtTraceMin}) and comes into play if we have to compute the normalization of elliptic elements.
It uses an auxiliary algorithm, called Rational Angle Recovery Algorithm (see Algorithm~\ref{alg:RatAngle}),
which checks whether the trace~$x$ of an elliptic element of~$\PSLR$ is of the form $x = 2\,\cos(p\pi/q)$
with $q\ge 2$, $1\le p\le q-1$, and $\gcd(p,q)=1$, and recovers the pair $(p,q)$ in this case.
For the explicit formulations of these algorithms we refer the reader to Section~\ref{sec3}.

The answer to the second question above can be found using the Trace Minimization Algorithm.
Geometric conditions for $\langle R,S\rangle$ to be a free Fuchsian group of rank~2 were
given by J.\ Gilman in~\cite{G3}. An algorithmic criterion is given as follows.

\begin{theorem}[Roots Generating Free Fuchsian Groups]\label{thm:free}\, \\
Suppose that we are given the trace triple $(\tr(A),\tr(B),\tr(AB))$ of a
free Fuchsian group $\langle A,B\rangle$ of rank~2. Let $R,S \in \PSLR$ with $R^m=A$ and
$S^n=B$ for some $m,n\in \mathbb{N}_+$ Compute the trace triple $(x,y,z) = (\tr(R),\tr(S),\tr(RS))$ using
Theorem~\ref{thm:CompTraceTriple}. Then the group $\langle R,S\rangle$ is discrete
and free of rank~2 if and only if one of the following two cases occurs.

\medskip
\noindent{\rm\bf Case F1:} $\tr([A,B]) \le -2$ and
$$
S_m(x)^2 \cdot S_n(y)^2 \;\le\; \tfrac{1}{2} - \tfrac{1}{4}\, 
\tr([A,B]).
$$

\medskip
\noindent{\rm\bf Case F2:} $\tr([A,B]) > 2$ and the trace triple $(x',y',z')$
which results from applying the Trace Minimization Algorithm~\ref{alg:TraceMin}
to $(x,y,z)$ satisfies $x'\ge 2$, $y'\ge 2$, and $z'\le -2$.
\end{theorem}

Next we present algorithmic answers to the first question, i.e., to the question when 
roots of the generators of a Fuchsian group generate a Fuchsian group.

Let $A,B\in \PSLR$ be elements which generate a Fuchsian group, 
let $R,S\in \PSLR$ such that $R^m=A$ and $S^n=B$ for some $m,n\in\mathbb{N}_+$. 
Our algorithms depend on the trace triple of $(A,B)$ and the commutator trace $\tr([A,B])$.
The first theorem relates to the case when~$A$ and~$B$ are not elliptic and
$\langle A,B\rangle$ is not free.

\begin{theorem}[The Torsion Case with Non-Elliptic Generators]\label{thm:torsion}\, \\
Suppose that we are given the trace triple $(\tr(A),\tr(B),\tr(AB))$ of a
Fuchsian group $\langle A,B\rangle$ which contains an elliptic element.
Compute the trace triple $(x,y,z) = (\tr(R),\tr(S),\tr(RS))$ using
Theorem~\ref{thm:CompTraceTriple}. 

\medskip
\noindent{\rm\bf Case I:} $-2 < \tr([A,B]) < 2$. (Note that this forces $A,B$ to be hyperbolic elements.)
Then the group $\langle R,S\rangle$ is discrete if and only if $m=n=1$.

\medskip
\noindent{\rm\bf  Case II:} $\tr([A,B])>2$, $\tr(A)\ge 2$, and $\tr(B)\ge 2$. 
Apply the Extended Trace Minimizing Algorithm~\ref{alg:ExtTraceMin} 
to the triple $(x,y,z)$, and let $(x',y',z')$ be the trace triple it returns.
Then the group $\langle R,S\rangle$ is discrete if and only if one of the
following cases occurs:
\begin{enumerate}
\item[(a)] $x'= 2\, \cos(\pi/p)$ for some $p\ge 2$ and $y'\ge 2$, as well as $z'\le -2$.

In this case, $\langle R,S\rangle$ is isomorphic to $\ZZ_p \ast \ZZ$.

\item[(b)] $x'= 2\, \cos(\pi/p)$ for some $p\ge 2$ and
$y' = 2 \, \cos(\pi/q)$ for some $q\ge p$, as well as $z'\le -2$.
In this case, $\langle R,S\rangle$ is isomorphic to $\ZZ_p \ast \ZZ_q$.

\item[(c)] $x' = 2\, \cos(\pi/p)$ for some $p\ge 2$ and $y'= 2\, \cos(\pi/q)$
for some $q\ge p$, as well as $z' = -2 \, \cos(r\pi/s)$ for some $s\ge 2$ and $r\in \{1,\dots,s-1\}$
with $\gcd(r,s)=1$, such that one of the following conditions holds:
\begin{itemize}
\item[(i)] $r=1$

\item[(ii)] $p=q$, $r=2$, and $s\ge 3$ odd.

\item[(iii)] $p=2$, $q\ge 3$ odd, $r=2$, $s=q$.

\item[(iv)] $p=3$, $q\ge 7$, $\gcd(q,3)=1$, $r=3$, $s=q$.

\item[(v)] $p=q=s$, $r=4$, $s\ge 7$ odd.

\item[(vi)] $p=3$, $q=7$, $r=2$, $s=7$
\end{itemize}
In these cases, $\langle R,S \rangle$ is a triangle group.

\end{enumerate}
\end{theorem}

Here the Extended Trace Minimization Algorithm applies to groups $\langle A,B\rangle$
which contain elliptic elements. It extends the usual Trace Minimization Algorithm by
normalization steps for elliptic elements.
 
Now we examine the case in which one or both generators $A,B$ of the 
Fuchsian group $\langle A,B\rangle$ are elliptic. 
If~$A$ is elliptic, there may be several elements $R\in \PSLR$ such that $R^m=A$. 
The analogous statement holds if~$B$ is elliptic.
The decision whether $\langle R,S\rangle$ is discrete or not may depend on which root we choose 
(see Example~\ref{ex:depends}).
Hence the computation of the trace triple of $(R,S)$ requires an additional input which specifies which root
of~$A$ (and possibly of~$B$) we want to consider.

\begin{theorem}[The Torsion Case with an Elliptic Generator]\label{thm:TorsionElliptic}\, \\
Suppose that we are given the trace triple $(\tr(A),\tr(B),\tr(AB))$ of a
Fuchsian group $\langle A,B\rangle$ such that~$A$ is an elliptic element.

\medskip
\noindent{\bf Case III:} Perform the following steps:
\begin{enumerate}
\item[(1)] Choose numbers $k\in\{0,\dots,m-1\}$ (and $\ell\in \{0,\dots,n-1\}$ if~$B$ is elliptic)
which determine~$R$ (and~$S$ if~$B$ is elliptic). 

\item[(2)] Compute the trace triple $(x,y,z)$ of $(R,S)$.

\item[(3)] Apply the Extended Trace Minimizing Algorithm~\ref{alg:ExtTraceMin} 
to $(x,y,z)$, and let $(x',y',z')$ be the trace triple it returns.
\end{enumerate}

Then the group $\langle R,S\rangle$ is discrete if and only if $(x',y',z')$
satisfies one of conditions in Case~II.
\end{theorem}

The remaining situation is when $\langle A,B\rangle$ is a free Fuchsian group.

\begin{theorem}[The Free Case]\label{thm:CaseIV}\, \\
Suppose that we are given the trace triple $(\tr(A),\tr(B),\tr(AB))$ of a free
Fuchsian group $\langle A,B\rangle$ of rank~2.

\medskip
\noindent{\bf Case IV:} Perform the following steps.
\begin{enumerate}
\item[(1)] Compute the trace triple $(x,y,z)$ of~$(R,S)$ and $\tau=\tr([R,S])$.

\item[(2)] Apply the Trace Minimization Algorithm~\ref{alg:TraceMin} to
$(x,y,z)$, and let $(x',y',z')$ be its result. 
\end{enumerate}

Then $\langle R,S\rangle$ is a discrete group if and only if one of the following cases
occurs.
\begin{enumerate}
\item[(a)] Using Theorem~\ref{thm:free}, the group $\langle R,S\rangle$ 
turns out to be a free Fuchsian group of rank~2.

\item[(b)] $-2 < \tau < 2$ and the Rational Angle Recovery Algorithm, applied to~$\vert \tau\vert$, 
yields a pair $(p,q)$ such that one of the following six cases holds:
\begin{itemize}
\item[(i)] $p=1$, $q\ge 2$.

\item[(ii)] $p=2$, $q\ge 3$ odd.

\item[(iii)] $p=6$, $q\ge 7$, $\gcd(q,6)=1$, $x'=y'=z'$.

\item[(iv)] $p=4$, $q\ge 5$ odd, $x'=y'$, $z'=\frac{1}{2}\, (x')^2$.

\item[(v)] $p=3$, $q\ge 4$, $\gcd(q,3)=1$, $x'=y'=z'$.

\item[(vi)] $p=4$, $q=7$, $x'+1 = y' = z'$. 
\end{itemize}

\item[(c)] $\langle R,S\rangle$ is not free, $\tau > 2$, and
the Extended Trace Minimization Algorithm~\ref{alg:ExtTraceMin}, applied to $(x,y,z)$,
yields a triple $(x'',y'',z'')$ which satisfies one of conditions in Case~II.

\end{enumerate}
\end{theorem}

Here $\langle R,S\rangle$ is a triangle group in cases (b.iii) - (b.vi).

\medskip
Let us give a brief outline of the paper.
In Section~\ref{sec2} we recall some basic results about $\PSLR$ and its subgroups. 
In addition, we collect a number of formulas for the traces of elements of $\PSLR$
and for Chebyshev S-polynomials. The main tools used later on are the power formula 
$$
A^n \;=\; S_n(\tr(A))\cdot A - S_{n-1}(\tr(A))\cdot E
$$
and the trace formula for commutators of powers
$$
\tr([A^m,B^n])-2 \;=\; S_m(\tr(A))^2 \cdot S_n(\tr(B))^2 \cdot (\tr([A,B])-2)
$$
for elements $A,B\in \PSLR$. To keep the paper reasonably self-contained and complete,
we include proofs for those formulas which were not readily available in the literature.
Consequently, we get a proof of Theorem~\ref{thm:CompTraceTriple}.
Similar formulas using the related numbers $\beta(A)=\tr(A)^2-4$ and $\gamma(A,B)=\tr([A,B])-2$
appear in the works of F.W.\ Gehring and G.J.\ Martin~\cite{GMa,Mar}. 

In Section~\ref{sec3} we first recall some basics about Nielsen transformations and Nielsen equivalence.
Then we formulate and prove the Trace Minimization Algorithm
for pairs of elements $(A,B)$ of~$\PSLR$ with $\tr([A,B])\ne 2$ (see Alg.~\ref{alg:TraceMin}
and Theorem~\ref{thm:TraceMin}). Although this technique has been used numerous times
in the past (see e.g.~\cite{FR,KIR,Ros2,Ros3}), apparently no general, explicit formulation has been available
until now. Here we fill this gap and also provide expicitly computed examples (see Examples~\ref{ex:NegTau}
and~\ref{ex:PosTau}).

Then we turn our attention to the Extended Trace Minimization Algorithm (see Algorithm~\ref{alg:ExtTraceMin}
and Theorem~\ref{thm:ExtTraceMin}).
It needs several ingredients, most notably the Rational Angle Recovery Algorithm (see Algorithm~\ref{alg:RatAngle}
and Theorem~\ref{thm:RatAngleAlg})
and an explicit way to perform normalizations of elliptic elements of finite order.

At this point we are ready to prove Theorem~\ref{thm:free} in Section~\ref{sec4}. Besides Theorem~\ref{thm:CompTraceTriple},
the main ingredient here is  the classification of 2-generated free Fuchsian groups of rank~2 (see Theorem~\ref{thm:classify}).

The question when groups $\langle R,S\rangle$ containing torsion are discrete is treated in Section~\ref{sec5}.
If the commutator $[A,B]$ is a torsion element, we prove Theorem~\ref{thm:EllipticComm} which says that
$\langle R,S\rangle$ can only be discrete for $m=n=1$. This proof uses the classification of 
2-generated Fuchsian groups with elliptic commutator (see Theorem~\ref{thm:ClassifyElliptic})
to conclude that $\tr([A,B])\le 1$.

The next step is the proof of Theorem~\ref{thm:torsion} which treats the case when $\tr([A,B])>2$ and $A,B$
are non-elliptic, but generate a Fuchsian group with torsion. In this case we need to apply the Extended 
Trace Minimization Algorithm~\ref{alg:ExtTraceMin}. Then we base the remaining steps of the proof on the
general classification of 2-generated Fuchsian groups with elliptic elements (see Theorem~\ref{thm:TheoremC}) 
and on Knapp's classic classification of Fuchsian groups with two elliptic generators (see~Theorem~\ref{thm:Knapp}). 

For the proof of Theorem~\ref{thm:TorsionElliptic}, we use the Extended Trace Minimization Algorithm~\ref{alg:ExtTraceMin}
again. Here one or both of the generators $A,B$ of the given Fuchsian group are assumed to be elliptic.
Thus the roots are in general not unique, and the answer to the question whether $\langle R,S\rangle$
is discrete may depend on the choice of a root (see Example~\ref{ex:depends}).

The last proof in this section relates to the case when $\langle A,B\rangle$ is free. It turns out that
the previous classification results, together with our Algorithms~\ref{alg:TraceMin} and~\ref{alg:ExtTraceMin},
are sufficient to achieve the desired result.

Finally, we present an extension of these algorithms to the case of rational powers 
$R^{\,p/q}=A$ and $S^{\,p'/q'}=B$. This uses a straightforward application of our trace formulas
to compute the trace triple of $(A^q,B^{\,q'})$.

Since the literature on the entire topic is quite vast, we tried to keep this paper as
self-contained as possible and to spell out the classification results we use explicitly 
and in our notation.
All group theoretic definitions and statements used without reference can be found
in~\cite{FMRSW} and~\cite{FR}.

%
%

\section{Preliminaries About $\PSLR$}
\label {sec2}

In this section we collect some basic facts about $\PSLR$, its subgroups and
related objects that we will need later.
Recall that an element~$A$ of the projective special linear group
$\PSLR = \SLR / \{ I_2,-I_2 \}$ can be considered as a fractional linear
transformation $z \mapsto \frac{az+b}{cz+d}$, where $A= \{ \overline{A},
-\overline{A} \}$ is the residue class of $\overline{A} = \binom{a\; b}{c\;d}$
with $ad-bc=1$. Here $I_2 = \binom{1\;0}{0\;1}$ denotes the identity matrix of 
size $2\times 2$.

For an element $A = \{ \overline{A},
-\overline{A} \}$ of~$\PSLR$, the trace of~$A$ is defined only up to sign, but
$\vert \tr(A) \vert$ is well-defined. Notice, however, that for a (multiplicative) commutator
$[A,B] = A\, B\, A^{-1}\, B^{-1}$ of elements $A,B\in \PSLR$, the trace $\tr([A,B])$
is well-defined, as this commutator has a unique representative in~$\SLR$.
By~$E$ we denote the identity element $\{I_2,-I_2\}$ of~$\PSLR$.

\begin{definition}
For $A\in \PSLR \setminus \{E\}$, we say that
\begin{enumerate}
\item[(a)] the element $A$ is \textbf{hyperbolic} if $\vert \tr(A) \vert >2$,

\item[(b)] the element $A$ is \textbf{parabolic} if $\vert \tr(A) \vert =2$, and

\item[(c)] the element $A$ is \textbf{elliptic} if $\vert \tr(A) \vert <2$.

\end{enumerate}
\end{definition}

Furthermore, an element~$A$ of $\PSLR$ has finite order if and only if 
$\vert \tr(A) \vert = 2 \cos(p\,\pi/q)$ for some $p,q\in\mathbb{N}_+$ with $1\le p < q$ and $\gcd(p,q)=1$.

\medskip
The following properties of subgroups of $\PSLR$ play a role in this paper.

\begin{definition}
Let $G$ be a subgroup of $\PSLR$.
\begin{enumerate}
\item[(a)] The group~$G$ is called \textbf{discrete} if it does not contain any 
convergent sequence of pairwise distinct elements.

\item[(b)] The group~$G$ is called \textbf{elementary} if the commutator of any
two elements of infinite order has trace~2.

\item[(c)] A non-elementary discrete subgroup of~$\PSLR$ is called a \textbf{Fuchsian group}.

\item[(d)] The group~$G$ is called \textbf{elliptic} if every element $A\in G \setminus \{E\}$
is elliptic.

\end{enumerate}
\end{definition}

Notice that Fuchsian groups are sometimes defined by condition~(a) only.
In this paper we consider solely non-elementary Fuchsian groups. 
In the language of fractional transformations, the group~$G$ is elementary
if and only if two elements of infinite order have at least one common fixed point.

For subgroups $G=\langle A,B\rangle$ generated by two elements of $\PSLR$,
we choose representatives~$\widehat{A}$ of~$A$ and $\widehat{B}$ of~$B$ in $\SLR$ 
such that $\tr(\widehat{A}) \ge 0$ and $\tr(\widehat{B})\ge 0$. For every word~$C$ in~$A$
and~$B$, we then let $\tr(C)$ be the trace of the corresponding word in~$\widehat{A}$
and~$\widehat{B}$. The traces of elements of $\PSLR$ satisfy the following properties.

\begin{proposition}\label{prop:traceprops}
Let $A,B\in \PSLR$.
\begin{enumerate}
\item[(a)] $\tr(A^{-1}) = \tr(A)$ 

\item[(b)] $\tr(BAB^{-1}) = \tr(A)$ and $\tr(AB) = \tr(BA)$

\item[(c)] $\tr(AB^{-1}) = \tr(A) \cdot \tr(B) - \tr(AB)$ 

\item[(d)] $\tr([A,B]) =  \tr(A)^2 + \tr(B)^2 + (\tr(AB))^2 - \tr(A) \cdot \tr(B) \cdot \tr(AB) -2$

\item[(e)] If $\vert \tr(A) \vert \le 2$ then $\tr([A,B]) \ge 2$.

\item[(f)] If $\tr([A,B]) < 2$ then $\vert \tr(A)\vert > 2$ and $\vert \tr(B)\vert > 2$.

\end{enumerate}
\end{proposition}

\begin{proof}
Formulas (a) and~(b) are classically known and can be verified by direct calculation.
Formulas~(c) and~(d) are also well-known and are, for instance, shown in~\cite[Lemma 6]{Ros1}.
Claim~(e) was shown in~\cite[Lemma 1]{FR}. Claim~(f) follows immediately from~(e).
\end{proof}

For more complicated trace formulas, we need the following polynomials.

\begin{definition}
In the polynomial ring $\mathbb{R}[x]$, we recursively define polynomials
$S_0(x)=0$, $S_1(x)=1$, and
$$
S_n(x) \;=\; x\cdot S_{n-1}(x) - S_{n-2}(x)
$$
for $n\ge 2$. Moreover, for $m<0$, we let $S_m(x) = -S_{-m}(x)$.
Then the polynomial $S_n(x)$ is called the $n$-th \textbf{Chebyshev S-polynomial}
or the \textbf{scaled Chebyshev polynomial of the second kind}.
\end{definition}

The Chebyshev S-polynomials are related to the usual Chebyshev polynomials 
of the second kind $U_n(x)$ by the formula $S_n(x)=U_{n-1}(x/2)$. 
They have the following basic properties which can be derived 
from the analogous formulas for the polynomials $U_n(x)$ 
(see for instance~\cite{AS}, \cite{MH}, \cite{Sny}).

\begin{proposition}\label{prop:SnRules}
Let $m,n\in\mathbb{N}$. Then the following formulas hold true.
\begin{enumerate}
\item[(a)] For $\phi\in ]0,\frac{\pi}{2}]$, we have
$S_n(2\cos(\phi)) = \frac{\sin(n\phi)}{\sin(\phi)}$.

\item[(b)] $S_{m+n}(x) = S_m(x) S_{n+1}(x) - S_{m-1}(x)S_n(x)$

\item[(c)] $S_n(x)^2 - S_{n+1}(x) S_{n-1}(x) = 1$

\item[(d)] $S_{mn}(x) = S_m(S_{n+1}(x)-S_{n-1}(x))\cdot S_n(x)$

\item[(e)] $S_n(x + \frac{1}{x}) = \frac{1-x^{2n}}{x^{n-1}(1-x^2)}$

\item[(f)] $S_n(2)=n$

\end{enumerate}
\end{proposition}

\begin{proof}
The first formula follows, for instance, from~\cite[Formula 22.3.16]{AS}.

Formula~(b) follows from the addition theorem for the sin function.
More precisely, let $x=2\cos(\phi)$ and use
$$
S_{m+n}(x)\cdot \sin(\phi) = \sin((m+n)\phi) = \sin(m\phi) \cos(n\phi)
+ \cos(m\phi)\sin(n\phi)
$$
When we multiply the right-hand side of~(b) by $\sin(\phi)$ and use the addition theorem
on~$\sin(m\phi-\phi)$ and $\sin(n\phi+\phi)$, we obtain the same result.

To prove~(c), we apply the recursion formula and get
\begin{gather*}
S_n(x)^2 - S_{n+1}(x)S_{n-1}(x) = S_n(x)(xS_{n-1}(x) - S_{n-2}(x))\qquad\qquad\qquad \\ 
\qquad\qquad\qquad\qquad - (xS_n(x)-S_{n-1}(x))S_{n-1}(x) = S_{n-1}(x)^2 - S_n(x)S_{n-2}(x)
\end{gather*}

Formula~(d) is a consequence of the well-known composition formula 
$U_{mn-1}(x) = U_{m-1}(T_n(x))\cdot U_{n-1}(x)$ for Chebyshev polynomials
of the first and second kind (cf.~\cite[Ex.~1.5.3]{MH}).

Claim~(e) follows from $U_{n-1}(\frac{1}{2}(x+x^{-1})) = \frac{x^n-x^{-n}}{x-x^{-1}}$
(cf.~\cite[Formula~1.5.1]{MH}). 

Finally, the value in~(f) is for instance given 
in~\cite[Formula 22.2.7]{AS}. (However, note the index shift in our definition of $S_n(x)$.)
\end{proof}

\begin{proposition}\label{prop:MoreSnProps}
Let $n\in\mathbb{N}$, $\phi\in \mathbb{R}$, and $a\in\mathbb{R}$ with $a>2$. 
Then the polynomials $S_n(x)$ have the following properties.
\begin{enumerate}
\item[(a)] $S_n(2\cosh(\phi)) = \frac{\sinh(n\phi)}{\sinh(\phi)}$ for every $\phi\in\mathbb{R}$.

\item[(b)] $S_n(a) = \frac{(a+\sqrt{a^2-4})^n - (a+\sqrt{a^2-4})^{-n}}{2\sqrt{a^2-4}}$ 

\item[(c)] $\lim_{n\to\infty} \frac{S_{n+1}(a)}{S_n(a)} = \frac{1}{2} (a+ \sqrt{a^2-4})$ 

\item[(d)] $S_n(a) > n$
\end{enumerate}
\end{proposition}

\begin{proof}
The first formula is obtained by setting $\phi=ix$ in
$S_n(2\cos(\phi)) = \frac{\sin(n\phi)}{\sin(\phi)}$ and using $\cosh(x)=\cos(ix)$
as well as $\sinh(x)=-i\sin(ix)$.

Formula~(b) is a version of the analogous formula for $U_n(x)$, cf.~\cite[Formula 1.5.2]{MH}.
Next we note that the limit in~(c) follows readily from~(b).

Finally, we show~(d). Since $a>2$, we can write $a = 2 \cosh(\phi)$ with $\phi>0$ and use~(a).
It follows that $S_n(a)$ is strictly increasing for $a\ge 2$. As Proposition~\ref{prop:SnRules}.f
yields $S_n(2)=n$, the claim follows.
\end{proof}

Using the Chebyshev S-polynomials, we can prove the following formulas
for powers and roots in $\PSLR$.

\begin{proposition}[\textbf{Power Formulas in $\PSLR$}]\label{prop:power}\, \\
Let $m,n\ge 0$, and let $A,B\in \PSLR$. Then the following formulas hold.
\begin{enumerate}
\item[(a)] $A^n = S_n(\tr(A)) \cdot A - S_{n-1}(\tr(A))\cdot E$

\item[(b)] $\tr(A^n) = \tr(A)\, S_n(\tr(A)) - 2\, S_{n-1}(\tr(A))$

\item[(c)] $\tr([A^m,B^n]) - 2 = S_m(\tr(A))^2 \cdot S_n(\tr(B))^2 \cdot (\tr([A,B] - 2)$

\item[(d)] $\tr(A^n B) = S_n(\tr(A)) \cdot \tr(AB) - S_{n-1}(\tr(A))\cdot \tr(B)$ for $n\ge 1$.

\end{enumerate}
\end{proposition}

\begin{proof}
First we prove~(a) by induction on~$n$. For $n=1$, this is clear. For $n=2$, we have to show
$A^2 = \tr(A)\cdot A - E$. This follows from the Cayley-Hamilton Theorem and $\det(A)=1$.
For $n\ge 3$, we use the recursive formula for $S_n(x)$ and the case $n=2$ to compute
\begin{align*}
A^n &\;=\; A\, (S_{n-1}(\tr(A))\cdot A - S_{n-2}(\tr(A))\, E)\\ 
&\;=\; S_{n-1}(\tr(A))\, (\tr(A)\,A-E) - S_{n-2}(\tr(A))\, A\\  
&\;=\; (\tr(A)\,S_{n-1}(\tr(A))-S_{n-2}(\tr(A)))\, A - S_{n-1}(\tr(A))\, E\\
&\;=\; S_n(\tr(A))\, A - S_{n-1}(\tr(A))\, E
\end{align*}
Claim~(b) is an immediate consequence of~(a).

Thus we show~(c) next. It suffices to treat the case $m=1$, since then
\begin{align*}
\tr([A^m,B^n])-2 &\;=\; S_n(\tr(B))^2\, (\tr([A^m,B])-2) &\\ 
&\;=\; S_n(\tr(B))^2\, (\tr([B,A^m])-2) &\\ 
&\;=\; S_n(\tr(B))^2 S_m(\tr(A))^2\, (\tr([A,B])-2) &
\end{align*}
Thus we claim that $\tr([A,B^n])-2 = S_n(\tr(B))^2(\tr([A,B])-2)$. To ease the notation,
we let $x=\tr(A)$, $y=\tr(B)$, and $s_i = S_i(\tr(B))$ for $i\in\mathbb{N}$.
Using Proposition~\ref{prop:traceprops} and~(a), we calculate
$$
\tr(AB^n) = \tr(A\, (s_n\,B - s_{n-1}\, E)) = s_n\, \tr(AB) - s_{n-1}\, x
$$
and $\tr(AB^{-n}) = \tr(A^{-1}B^n) = s_n\, \tr(A^{-1}B) - s_{n-1}\, x$. This yields
\begin{align*}
\tr(&[A,B^n]) \;=\; x^2 +\tr(B^n)^2 - \tr(AB^n)\tr(AB^{-n})-4 &\\
&=\; x^2+(s_n y-2s_{n-1})^2 - (s_n\tr(AB)-s_{n-1}x)(s_n\tr(A^{-1}B) - s_{n-1}x) {-}\,4 &\\
&=\; x^2 +s_n^2 y^2 - 4 s_ns_{n-1}y + 4s_{n-1}^2 - s_n^2 \tr(AB)\tr(A^{-1}B) &\\ 
& \qquad + s_n s_{n-1} x^2y - s_{n-1}^2 x^2 -4 &\\
&=\; (1+s_{n+1}s_{n-1})\, x^2 + s_n^2 y^2 - 4 s_{n+1}s_{n-1} -s_n^2 \tr(AB)\tr(A^{-1}B) - 4 &\\
&=\; s_n^2 x^2 + s_n^2 y^2 - s_n^2 \tr(AB)\tr(A^{-1}B) - 4 s_n^2 &\\
&=\; s_n^2\, (\tr([A,B]-2), 
\end{align*}
as claimed.

Finally, we prove~(d) by induction on~$n$. For $n=1$, the claim is clearly true.
For $n\ge 2$, we let $x=\tr(A)$ and calculate
\begin{align*}
\tr(A^n & B) \;=\; x\, \tr(A^{n-1}B) - \tr(A\, (A^{n-1}B)^{-1}) \\
\;=\;& x\, [ S_{n-1}(x) \tr(AB) - S_{n-2}(x) \tr(B)] - \tr(A^{n-1}\cdot BA^{-1}) \\
\;=\;& x S_{n-1}(x) \tr(AB) - x S_{n-2}(x) \tr(B) - S_{n-1}(x) \tr(B) + S_{n-2}(x) \tr(BA^{-1}) \\
\;=\;& \tr(AB)\, [x\, S_{n-1}(x) - S_{n-2}(x)] - S_{n-1}(x) \tr(B)\\ 
\;=\;& S_n(x)\tr(AB)  - S_{n-1}(x)\tr(B),
\end{align*}
as claimed.
\end{proof}

\begin{proposition}[\textbf{Root Formulas in $\PSLR$}]$\mathstrut$\label{prop:root}\\
Let $m,n\in\mathbb{N}_+$, and let $A,B,R,S\in\PSLR \setminus \{E\}$ such that $R^m=A$ and $S^n=B$.
For $i,j\in\mathbb{N}$, let $s_i = S_i(\tr(R))$ and $t_j=S_j(\tr(S))$.
Then the following claims hold.
\begin{enumerate}
\item[(a)] $S_k(\tr(R)) \ne 0$ if $\vert \tr(R) \vert \ge 2$, or if~$R$ has a finite order
which does not divide~$k$.

\item[(b)] $R = \frac{1}{s_m}\, (A + s_{m-1}E)$ if $s_m\ne 0$ and
$S = \frac{1}{t_n}\, (B + t_{n-1}E)$ if $t_n\ne 0$.

\item[(c)] $RS = \frac{1}{s_m t_n}\, (AB + t_{n-1}A + s_{m-1}B + s_{m-1}t_{n-1}E)$
if $s_m,t_n\ne 0$.

\item[(d)] $\tr(RS) = \frac{1}{s_m t_n}\, (\tr(AB) + s_{m+1}t_{n-1} + s_{m-1} t_{n+1})$
if $s_m,t_n\ne 0$.
\end{enumerate}
\end{proposition}

\begin{proof}
First we show (a). For $\tr(R) \ge 2$, this follows from Proposition~\ref{prop:MoreSnProps}.a.
For $\tr(R) \le -2$, we can use the fact that $S_k(x)$ is even or odd, whence
$S_k(\tr(A)) = \pm S_k(-\tr(A)) \ne 0$. Now assume that $R$ has a finite order $\ell=\ord(R)\ge 2$.
Up to conjugacy, the matrix~$R$ is the rotation matrix by some angle $\phi= p\pi / \ell$, where
$\gcd(p,\ell)=1$. Then $\tr(R) = 2 \cos(\phi)$ and $S_k(\tr(R)) = \sin(k\phi)/\sin(\phi)$ by
Proposition~\ref{prop:MoreSnProps}.a. Thus $S_k(\tr(R))=0$ if and only if $kp/\ell$
is an integer, i.e., if and only if~$\ell$ divides~$k$.

Claim~(b) follows immediately from Proposition~\ref{prop:power}.a
by replacing~$A$ with~$R$ and~$S$, respectively.
Formula~(c) follows by multiplying the two formulas in~(b). 

Finally, let us prove part~(d). We let $x=\tr(R)$ and $y=\tr(S)$ and apply the trace
map to~(c).
Moreover, we use the formulas $\tr(A) = \tr(R^m) = x\,s_m - 2 s_{m-1} = s_{m+1} - s_{m-1}$ as well as
$\tr(B) = \tr(S^n) = y\,t_n - 2 t_{n-1} = t_{n+1} - t_{n-1}$ which are a consequence
of the recursion formula for $S_i(x)$ and Proposition~\ref{prop:power}.a. The result follows readily. 
\end{proof}

Now we have all tools to give a proof of Theorem~\ref{thm:CompTraceTriple}. 
Let $A,B,R,S \in \PSLR$ such that $R^m=A$ and $S^n=B$ for some $m,n\in\mathbb{N}_+$, and
assume that we are given the trace triple $(\tr(A),\tr(B),\tr(AB))$ via algebraic numbers, i.e.,
via their minimal polynomials and isolating intervals.
Notice that parts~(a), (b), (c) of Theorem~\ref{thm:CompTraceTriple} allow us to calculate isolating intervals for
$x=\tr(R)$ and $y=\tr(S)$. Consequently, part~(d) yields the minimal polynomials of~$x$ and~$y$, and finally part~(e)
computes $z=\tr(RS)$.

\begin{proof}[{\bf Proof of Theorem~\ref{thm:CompTraceTriple}}]
Claim~(a) is a well-known property of parabolic elements and follows for instance
from Proposition~\ref{prop:power} together with Proposition~\ref{prop:SnRules}.f.

Next we show~(b). The number $\lambda = \frac{1}{2}(x+\sqrt{x^2-4})$ is the {\it multiplier}
of~$A$, i.e., its larger eigenvalue. Then the multiplier~$\mu$ of~$R$ satisfies $\mu^m=\lambda$,
and the trace of~$R$ is the sum of its two eigenvalues $\mu + \mu^{-1}$. This yields the stated formula.

The proof of~(c) proceeds analogously, but one has to take into account that there is
not a unique matrix~$R$. Instead, notice that $\lambda=e^{i\phi}$ is a complex number on the unit circle
and $\tr(A) = 2\,\cos(\phi)$ with $0<\phi\le \pi/2$. Then we obtain 
$x=\tr(R) = 2\,\cos(\frac{\phi}{m} + \ell\,\frac{\pi}{m})$ with $\ell\in \{0,\dots,m-1\}$, and this leads to
$\mu^m = (-1)^\ell\, \lambda$. Consequently, the imaginary parts in
$x = \mu + \mu^{-1}$ cancel, if we choose the two $m$-th roots to be complex conjugates.

Next, claim~(d) follows from Proposition~\ref{prop:power}.b. Plugging 
$x = x'\, S_m(x') - 2 S_{m-1}(x')$ into the minimal polynomial of~$x$ yields the claim.

Finally, the formula in~(e) is Proposition~\ref{prop:power}.c in the case
$m>1$, $n>1$, and it follows from Proposition~\ref{prop:power}.d in the 
other cases.
\end{proof}

%
%

\section{The Trace Minimization Algorithm}
\label{sec3}

The main tools in the various algorithms given in the introduction are the
Trace Minimization Algorithm and the Extended Trace Minimization Algorithm.
The trace minimization technique has been introduced by the third author in his
doctoral thesis~\cite{Ros0} whose key parts appeared also in~\cite{Ros1}.
Later it was used in various forms in a number of papers~\cite{KIR, Ros3,Ros4}, 
but apparently it has never been written up formally as an explicit algorithm. 
The present section is intended to remedy this situation. 

The Trace Minimization Algorithm is based on Nielsen transformations
and the concept of Nielsen equivalence of pairs of matrices which we recall first.

\begin{definition}
Let $A,B,C,D \in \PSLR$. The pair $(A,B)$ is called \textbf{Nielsen equivalent}
to the pair $(C,D)$ if there exists a regular Nielsen transformation from $(A,B)$
to $(C,D)$. A regular Nielsen transformation is a finite product of transformations 
of one of the following forms:
\begin{enumerate}
\item[(1)] Replace $(A,B)$ by $(B,A)$.
\item[(2)] Replace $(A,B)$ by $(A^{-1},B)$.
\item[(3)] Replace $(A,B)$ by $(A,BA)$.
\end{enumerate}
In this case we write $(A,B) \simN (C,D)$.
\end{definition}

These operations can also be combined in various ways.

\begin{remark}\label{rem:OtherNielsen}
By composing operations (1), (2), and (3), it follows that the following
operations are Nielsen transformations, too.
\begin{enumerate}
\item[(4)] Replace $(A,B)$ by $(A,BA^{-1})$.
\item[(5)] Replace $(A,B)$ by $(A,AB)$.
\item[(6)] Replace $(A,B)$ by $(A,A^{-1}B)$.
\end{enumerate}  
\end{remark}

Subsequently, the following invariant under Nielsen equivalence will be important.

\begin{proposition}\label{prop:trcomminv}
For $A,B,C,D \in \PSLR$ such that $(A,B) \simN (C,D)$, we have
$\tr([A,B]) = \tr([C,D])$.
\end{proposition}

\begin{proof}
Using Proposition~\ref{prop:traceprops}, we argue as follows.
The invariance under operation (1) follows from 
$$
\tr([B,A]) = \tr(BAB^{-1}A^{-1}) = \tr((BAB^{-1}A^{-1})^{-1}) = \tr(ABA^{-1}B^{-1}).
$$
The invariance under operation (2) follows from 
$$
\tr([A^{-1},B]) = \tr(A^{-1}(BAB^{-1}))
= \tr(BAB^{-1}A^{-1}) = \tr([B,A])=\tr([A,B]).
$$
The invariance under operation (3) follows from
\[
\tr([A,BA]) = \tr( A(BA)A^{-1}(A^{-1}B^{-1}) ) = \tr([A,B]). \qedhere
\]
\end{proof}

Now we fix two matrices $A,B\in \PSLR$. They generate a group $G=\langle A,B\rangle$.
Recall that, in~$\SLR$, we choose representatives $\widehat{A}$ and $\widehat{B}$
of~$A$ and~$B$, respectively, such that $\tr(\widehat{A})\ge 0$ and $\tr(\widehat{B})\ge 0$.
Then every element of~$G$ which is given as a word in~$A,B,A^{-1}, B^{-1}$ has a unique trace defined by
the trace of the representative given by the same word in~$\widehat{A}$ and $\widehat{B}$.

For a pair $(U,V) \in G^2$, we form its \textbf{trace triple}
$(\tr(U),\tr(V),\tr(UV))$. The following proposition describes the effect of Nielsen transformations
on these trace triples.

\begin{proposition}\label{prop:Replace}
Let $U,V \in G$, let $x=\tr(U)$, let $y=\tr(V)$, and let $z=\tr(UV)$.
\begin{enumerate}
\item[(a)] If we replace $(U,V)$ by $(V,U)$ then $(x,y,z)$ is replaced by $(y,x,z)$.

\item[(b)] If we replace $(U,V)$ by $(U^{-1},UV)$ then $(x,y,z)$ is replaced by $(x,z,y)$.

\item[(c)] If we replace $(U,V)$ by $(UV,V^{-1})$ then $(x,y,z)$ is replaced by $(z,y,x)$.

\item[(d)] If we replace $(U,V)$ by $(U^{-1},V)$ or $(U,V^{-1})$ then $(x,y,z)$ is 
replaced by $(x,y,xy-z)$.

\item[(e)] If we replace $(U,V)$ by $(U,U^{-1}V)$ then $(x,y,z)$ is replaced by $(x,xy-z,y)$.

\item[(f)] If we replace $(U,V)$ by $(UV^{-1},V)$ then $(x,y,z)$ is replaced by $(xy-z,y,x)$. 

\end{enumerate}
\end{proposition}

\begin{proof} 
This follows from Proposition~\ref{prop:traceprops}.c and Remark~\ref{rem:OtherNielsen}.
\end{proof}

Notice that all replacements of $(U,V)$ in this proposition correspond to Nielsen transformations.
Furthermore, the effects on the trace triples generate all permutations.

The following algorithm starts with the trace triple of a
pair of matrices $(A,B) \in \PSLR^2$. It computes the trace triple of a
pair of matrices $(U,V)$ which is Nielsen equivalent to $(A,B)$ 
and whose trace triple is in some sense minimal.

The case $\tr([A,B])=2$ has to be excluded here. In this case, which corresponds to two 
fractional linear transformations with a joint fixed point,
the group $G=\langle A,B\rangle$ is metabelian and not discrete, and hence elementary. 
These simple case has been excluded from our considerations.

\begin{algorithm}[ht]
  \DontPrintSemicolon
  \SetAlgoLongEnd
  \SetKwInOut{Input}{Input}
  \SetKwInOut{Output}{Output}

  \Input{A trace triple $(\tr(A),\tr(B),\tr(AB))$ of $(A,B)\in\PSLR^2$ such that $\tau=\tr([A,B])\ne 2$.}
  \Output{The trace triple of a pair $(U,V)\in \PSLR^2$.}
  \BlankLine
  Compute $\tau = \tr(A)^2 + \tr(B)^2 + \tr(AB)^2 -\tr(A)\tr(B)\tr(AB)-2$.\;
  Using Proposition~\ref{prop:Replace}.a,b,c, permute the given trace triple (and, if necessary, switch 
  two of the three signs) to get a triple
  $(x,y,z)$ which satisfies $2 < x\le y\le \vert z\vert $ in case $\tau<2$
  and $0\le x\le y\le \vert z\vert $ in case $\tau>2$.\;
  \Repeat{$z'\le \frac{1}{2}\, x'y'$ \text{\rm in the case} $\tau<2$, \text{\rm or} $x'<0$ 
  \text{\rm in the case} $\tau>2$.}{
  Using Proposition~\ref{prop:Replace}.e,f,
  compute a triple  $(x',y',z')$ which satisfies $x'=xy-z \le y' = y \le z'=x$ or
  $x'=x \le y'=xy-z \le z'=y$.
  }
  \Return $(x',y',z')$ in the case $\tau<2$, and $(y',z',x')$ in the case $\tau>2$.\;
  \caption{The Trace Minimization Algorithm}
  \label{alg:TraceMin}
\end{algorithm}

\begin{theorem}[{\bf Trace Minimization Algorithm}]$\mathstrut$\label{thm:TraceMin}\\
Algorithm~\ref{alg:TraceMin} computes the trace triple of a pair $(U,V) \in G^2$ 
which is Nielsen equivalent to $(A,B)$ and satisfies
\begin{enumerate}
\item[(a)] $2 < \tr(U) \le \tr(V) \le \tr(UV) \le \frac{1}{2}\, \tr(U)\tr(V)$
in the case $\tau<2$, or

\item[(b)] $0\le \tr(U) \le \tr(V)$ and $\tr(UV)<0$ in the case $\tau>2$.
\end{enumerate}
Moreover, in case~(a) the resulting trace triple is uniquely determined by the 
stated condition.
\end{theorem}

\begin{proof}
First we show that all steps can be executed. 
Proposition~\ref{prop:Replace}.a,b,c and $x\ge 0$, $y\ge 0$ 
imply that we can always achieve $0\le x\le y\le |z|$.
If necessary, we can switch the representative matrix of the first or second component
here which switches two signs.
In the case $\tau<2$, Propositions~\ref{prop:trcomminv} and~\ref{prop:traceprops}.f yield $x>2$.
Hence step~2 can be executed.

Next we have to verify that $xy-z<y$ in step~4, i.e., that we can arrange
$(x',y',z')$ in the desired way. For~$\tau<2$, this follows from~\cite[Prop.~2]{FR},
and for $\tau>2$, it is a consequence of~\cite[Prop.~4]{FR}.

In order to prove the finiteness of the algorithm, we have to show that
step~4 is performed only finitely often. This is shown in the proofs 
of Lemmas~2 and~3 of~\cite{FR}. Let us briefly sketch the idea. Starting from
$(x_1,y_1,z_1)=(x,y,z)$ with the trace triple from step~2, iterations
of step~4 yield a sequence $(x_i,y_i,z_i)$ with $x_{i+1}\le x_i$, $y_{i+1}
\le y_i$, and $z_{i+1}\le z_i$. If the termination condition is never met,
this sequence converges to a triple $(x_\infty, y_\infty, z_\infty)$.
Then one proves that this forces $y_\infty=z_\infty$ and $x_\infty=2$ which leads to a contradiction
in the formula $\tau = x_\infty^2 + y_\infty^2 + z_\infty^2 - x_\infty y_\infty z_\infty -2$.

The final item to check is the correctness of the algorithm, i.e., that~(a)
and~(b) are satisfied. This follows from the correctness of Proposition~\ref{prop:Replace}
and the termination conditions in line~(5). The additional claim is proven
in~\cite[Lemma~2]{FR}.
\end{proof}

\begin{remark}\label{rem:ComputeWords}
A simple extension of this algorithm keeps track of the Nielsen transformations
which are applied and computes words $u,v$ in the letters $A,B,A^{-1},B^{-1}$
such that pair of matrices $(U,V)$ which corresponds to the output of the Trace Minimization
Algorithm~\ref{alg:TraceMin} is $U=u(A,B,A^{-1},B^{-1})$ and $V=v(A,B,A^{-1},B^{-1})$.
\end{remark}

Let us illustrate the algorithm with a couple of examples.

\begin{example}\label{ex:NegTau}
Consider the trace triple $(87,6,507)$ which corresponds to
elements $A = \binom{44\;\;61}{31\;\;43}$ and
$B = \binom{3\;\;4}{2\;\;3}$ in $\PSLR$.
Let us follow the steps of the algorithm.
\begin{enumerate}
\item[(1)] $\tau = -2$

\item[(2)] $(x_1,y_1,z_1) = (6,87,507)$
corresp.\ to $(U_1,V_1,U_1V_1) = (B, A, BA)$.
 
\item[(4)] 
$(x_2,y_2,z_2)=(6,15,87)$ corresp.\ to $(U_2,V_2,U_2V_2)=(U_1^{-1},U_1V_1^{-1},V_1^{-1})$.
Here we have $U_2 = \binom{\;3\;\;-4}{-2\;\;\;3}$ and $V_2 = \binom{\;5\;\;-7}{-7\;\;10}$. 
 
\item[(4)] $(x_3,y_3,z_3)=(3,6,15)$ corresp.\ to $(U_3,V_3,U_3V_3)=(V_2 U_2^{-1}, U_2, V_2)$.
Here we have $U_3 = \binom{\;1\;\; -1}{-1\;\;\;2}$ and $V_3 = \binom{\;3\;\;-4}{-2\;\;\;3}$.

\item[(4)] $(x_4,y_4,z_4)=(3,3,6)$ corresp.\ to $(U_4,V_4,U_4V_4) = (U_3^{-1},U_3V_3^{-1},V_3^{-1})$.
Here we have $U_4 = \binom{2\;\;1}{1\;\;1}$ and $V_4 = \binom{1\;\;1}{1\;\;2}$.
 
\item[(5)] The next trace triple is $(3,3,3)$, where $V_5= U_4V_4^{-1} = \binom{\;3\;\;-1}{-1\;\;\;0}$ satisfies
$\tr(V_5)=3$. Since $3 < \frac{1}{2}\cdot 3\cdot 3 = \frac{9}{2}$,
the algorithm stops and returns $(3,3,3)$.
A corresponding pair of matrices is $(U_4^{-1},V_5) \simN (A,B)$.
\end{enumerate}
When we look at the last step, we see that also $U = \binom{2\;\;1}{1\;\;1}$
and $V = \binom{1\;\;1}{1\;\;2}$ define a pair $(U,V) \simN (A,B)$ which satisfies
$2\le \tr(U)=3\le \tr(V)=3 \le \tr(UV)=3 \le \frac{1}{2}\, \tr(U)\,\tr(V) = \frac{9}{2}$.
\end{example}

In the second example we encounter the case $\tau>2$.

\begin{example}\label{ex:PosTau}
Consider the triple $(26,53,5/2)$ which corresponds to the pair $(A,B)$
given by $A = \binom{26\;\;-1}{\;1 \;\;\;\; 0}$ and 
$B = \binom{\;0 \;\;\;\;\;2}{-1/2\;\; 53}$ in $\PSLR$.
Let us follow the steps of the algorithm.
\begin{enumerate}
\item[(1)] $\tau = 44.25$

\item[(2)] $(x_1,y_1,z_1)=(2.5,26,53)$ corresp.\ to $(U_1,V_1,U_1V_1)=(B^{-1}A^{-1},A,B^{-1})$.
           Here we have $U_1 = \binom{2\;\;\;\; 1}{\;0 \;\; 1/2}$ and $V_1=\binom{26\;\;-1}{\;1 \;\;\;\; 0}$.

\item[(4)] $(x_2,y_2,z_2) = (2.5,12,26)$ corresp.\ to
           $(U_2,V_2,U_2V_2)=(U_1^{-1}\!,U_1V_1^{-1}\!,V_1^{-1})$.
           Here we have $U_2 = \binom{1/2\;\;-1}{\;0 \;\;\;\; 2}$
           and $V_2 = \binom{\;-1\;\; 28}{-1/2\;\; 13}$.

\item[(4)] $(x_3,y_3,z_3) = (2.5,4,12)$ corresp.\ to
           $(U_3,V_3,U_3V_3)=(U_2^{-1},U_2V_2^{-1},V_2^{-1})$. 
           Here we have $U_3 = \binom{2\;\;\;\; 1}{\;0 \;\; 1/2}$
           and $V_3 = \binom{6\;\,-13}{1\;\;\;\;-2}$.

\item[(4)] $(x_4,y_4,z_4) = (-2,2.5,4)$ corresp.\ to $(U_4,V_4,U_4V_4) = (V_3U_3^{-1},U_3,V_3)$.
           Here we have $U_4 = \binom{\;3\;\;-32}{1/2\;\;-5}$
           and $V_4 = \binom{2\;\;\;\; 1}{\;0 \;\; 1/2}$.

\item[(5)] Return the tripel $(5/2,4,-2)$ which corresponds to 
           $(V_4,(U_4V_4)^{-1}) = ( \binom{2\;\;\;\; 1}{\;0 \;\; 1/2},\, 
           \binom{-2 \;\; 13}{-1 \;\;\;6})$.
\end{enumerate}

Thus the matrices $U= \binom{2\;\;\;\; 1}{\;0 \;\; 1/2}$ and  
$V = \binom{-2 \;\; 13}{-1 \;\;\;6})$ satisfy $(U,V) \simN (A,B)$ and
$0 \le \tr(U) = 5/2 \le \tr(V) = 4$ as well as $\tr(UV) = -2 < 0$.
\end{example}

Next we introduce the Extended Trace Minimization Algorithm. 
If we do not yet know whether a given subgroup of $\PSLR$ is discrete
or whether a given elliptic element $A\in \PSLR$ has finite order, we have to check
algorithmically whether $\tr(A)$ is of the form $\tr(A)=2\,\cos(p\,\pi/q)$ with
a rational number $0< p/q \le 1$. Since we assumed that $\tr(A)$ is given as an algebraic
number, we can use the following Algorithm~\ref{alg:RatAngle}. Here $\Phi(n)$ denotes
Euler's totient function.

\begin{algorithm}[ht]
  \DontPrintSemicolon
  \SetAlgoLongEnd
  \SetKwInOut{Input}{Input}
  \SetKwInOut{Output}{Output}

  \Input{An algebraic number $0\le x < 2$ given by its minimal polynomial $P(t)\in \mathbb{Q}[t]$
  and an isolating interval.}
  \Output{A pair $(p,q)$ with $q\ge 2$, $1\le p\le q-1$, $\gcd(p,q)=1$ if $x=2\,\cos(p\,\pi/q)$,
  and $(0,0)$ otherwise.}
  \BlankLine
  [Optional.] Use Sturm's Theorem to compute the number $\varrho$ of real roots of
  $P(t)$ in $[-2,2]$. If $\varrho\ne \deg(P(t))$ then \Return $(0,0)$.\;
  Compute all numbers $q_1,\dots,q_k$ such that $q_i\ge 2$ and $\Phi(2q_i)=2\,\deg(P(t))$.\;
  \For{$i=1,\dots,k$}{   
    \If{$P(t)$ {\rm divides} $S_{q_i}(t)$}{
    Compute the isolating interval of $p = q\, \arccos(x/2) / \pi$, and let $p$
    be the unique integer in this interval.\;
    \Return $(p,q_i)$.\;
    }
  }
  \Return $(0,0)$.\;
  \caption{The Rational Angle Recovery Algorithm}
  \label{alg:RatAngle}
\end{algorithm}

\begin{theorem}[{\bf Rational Angle Recovery Algorithm}]\label{thm:RatAngleAlg}\, \\
Let $0\le x <2$ be an algebraic number which is given by its minimal polynomial $P(t)\in\mathbb{Q}[t]$
and an isolating interval. Then Algorithm~\ref{alg:RatAngle} checks whether~$x$ is of the form
$x = 2\,\cos (p\,\pi/q)$ with a rational number $0< p/q \le 1/2$ and returns the pair $(p,q)$
with $q\ge 2$, $1\le p\le q-1$, $\gcd(p,q)=1$ in that case. If~$x$ is not of this form, the
algorithm returns $(0,0)$.
\end{theorem}

\begin{proof}
First of all, let us check the correctness of the optional step~1. 
As is well-known, a number of the form $x=2\,\cos(p\,\pi/q)$ is totally real and
all of its conjugates are in the interval $[-2,2]$ by a theorem of L.~Kronecker
(cf.~\cite{Pra}, Thm.~4.5.4). Hence Sturm's Theorem correctly checks if all zeros 
of $P(t)$ are in this interval.

Next we show that the algorithm is finite. The number $y=e^{ip\,\pi/q}$ is a primitive 
$2q$-th root of unity if $p$ is odd and a primitive $q$-th root of unity if $p$ is even.
Since $x=y+y^{-1}$, the field extension $\mathbb{Q}[y]/\mathbb{Q}[x]$ has degree~2, and thus
$\deg(P(t))= \Phi(2q)/2$. (Notice that $\Phi(2q)=\Phi(q)$ if $q$ is odd.)
Hence step~2 correctly determines all possible values of~$q$, and since $\Phi$ grows to infinity,
there are only finitely many such values.

Now we observe that $2\,\cos(p\,\pi/q)$ is a zero of~$S_{2q}(t)$ by Proposition~\ref{prop:SnRules}.a.
Hence $P(t)$ has to divide $S_{2q}(t)$ if~$x$ is of the desired form. If no polynomial $S_{2q_i}(t)$
is divisible by $P(t)$ for the possible values $q_1,\dots,q_k$ of~$q$, then~$x$ cannot be of the desired form.
This proves that the algorithm finds the correct value of the denominator~$q$ if it exists.
The unique value of~$p$ is then found using its isolating interval in step~5. 
\end{proof}

In this algorithm, the isolating interval of~$x$ has to be small enough, of course.
If it turns out that it isn't, we may have to shorten it using a root finding algorithm for~$P(t)$
such as Newton's method.

A further crucial ingredient for the Extended Trace Minimization Algorithm is 
the following normalization step.

\begin{proposition}[The Normalization Step]\label{prop:NormStep}\, \\
Let $U,V\in \PSLR$ such that $0<\tr(U)<2$ and $\tr(V)> 0$.
Assume that~$U$ has finite order and write $\tr(U) = 2 \cos(p\pi / q)$
with $q = \ord(U)\ge 3$ and $1\le p \le q-1$ such that $\gcd(p,q)=1$.
Write $rp-sq=1$ with $2\le r <q$ and $s\ge 1$. Let $U'=U^r$ and consider
the pair $(U',V)$.
\begin{enumerate}
\item[(a)] $\langle U',V \rangle = \langle U,V\rangle$

\item[(b)] Let $(x,y,z) = (\tr(U),\tr(V),\tr(UV))$ be the trace triple of~$(U,V)$.
Then the trace triple of $(U',V)$ is
$$
( x\, S_r(x) - 2\, S_{r-1}(x),\; y,\; S_r(x)\, z - S_{r-1}(x)\,y).
$$

\item[(c)] $\tr(U')= 2 \cos(\pi/q)$ and $q=\ord(U')$

\item[(d)] If $\tr([U,V])>2$ then $2<\tr([U',V]) \le \tr([U,V])$

\end{enumerate}
\end{proposition}

\begin{proof}
Note that $0<\tr(U)<2$ implies $\ord(U)\ge 3$. An elliptic element of some order $q\ge 3$
in $\PSLR$ is conjugate to a rotation by an angle $\frac{p\,\pi}{q}$, where
$1\le p \le q-1$ and $\gcd(p,q)=1$. Thus the trace of~$U$ is $2\cos(p\,\pi/q)$.

Claim~(a) follows from $(U')^p = U^{1+sq} = U$. The formula for $\tr(U')$ in~(c) is a
consequence of Proposition~\ref{prop:power}.a, and the formula for $\tr(U'V)=\tr(U^rV)$
was shown in Proposition~\ref{prop:power}.d.

Next, we note that claim~(c) follows from the observation that $U'$ is conjugate to the
$r$-th power of the rotation which is conjugate to~$U$.

Hence it remains to prove~(d). Since Proposition~\ref{prop:MoreSnProps}.d yields 
$$
\tr([U^r,V]) -2 \;=\; 
S_r(\tr(U))^2 \, (\tr([U,V])-2),
$$
it suffices to show that $S_r(\tr(U))^2 \le 1$. Here we apply Proposition~\ref{prop:SnRules}.a and get
$$
S_r(\tr(U)) \;=\; \tfrac{\sin(rp\,\pi / q)}{\sin(p\,\pi/q)} \;=\; \tfrac{(-1)^s \sin(\pi/q)}{\sin(p\,\pi/q)} 
$$
where the second equality follows from $rp = 1+sq$. Finally, we note that $1\le p \le q-1$ implies
$\sin(p\,\pi/q)\ge \sin(\pi/q)$ and obtain $S_r(\tr(U))^2 \le 1$, as desired.
\end{proof}

The passage from $(U,V)$ to $(U',V)$ in this proposition will be called a 
\textbf{normalization step}. More generally, a passage from a pair $(U,V)$ to
a pair $(U^r,V)$ has sometimes been called an \textbf{E-transformation} or an \textbf{extended 
Nielsen transformation}. If $p=1$, i.e., if $\vert \tr(U)\vert = 2\cos(\pi/q)$, we say that~$U$
(or $\tr(U)$) is \textbf{normalized}.

\begin{remark}
Let $U\in \PSLR$ be an elliptic element of finite order, and let $x = \vert \tr(U)\vert$. 
Then we can compute the exponent $r\ge 1$ such that $U^r$ is normalized as follows.
\begin{enumerate}
\item[(1)] Apply Algorithm~\ref{alg:RatAngle} to get a pair $(p,q)$. 

\item[(2)] If $p=1$ then return $r=1$.

\item[(3)] If $p>1$ then apply the Extended Euclidean Algorithm and find $r\ge 1$ and $s\ge 0$ such that $rp-sq=1$.
Return~$r$.
\end{enumerate}
Below we assume that a function ${\tt NormExp}(x)$ has been implemented which computes~$r$ if $p>1$.
\end{remark}

To simplify the presentation, we assume that, whenever we permute elements of the triple $(x,y,z)$, 
we choose the matrices whose traces are the new numbers~$x$ and~$y$ in such a way that $x\ge 0$ and $y\ge 0$. 

\begin{algorithm}[H]
  \DontPrintSemicolon
  \SetAlgoLongEnd
  \SetKwInOut{Input}{Input}
  \SetKwInOut{Output}{Output}

  \Input{A trace triple $(\tr(A),\tr(B),\tr([A,B]))$ of $(A,B)\in\PSLR^2$ such that 
  $\langle A,B\rangle$ contains elliptic elements, and such that $\tau=\tr([A,B])> 2$.}
  \Output{{\tt "not discrete"} or a triple $(x,y,z)$.}
  \BlankLine
  Using Algorithm~\ref{alg:TraceMin}, compute $(x,y,z)$ with $0\le x\le y$, $-2<z<0$.\;
  Permute $x,y,z$ such that $0\le x\le y \le \vert z\vert$.\;
  Apply Algorithm~\ref{alg:RatAngle} to~$x$ and get a pair $(p,q)$.\;
  \lIf{$(p,q)=(0,0)$}{\Return {\tt "not discrete"}.}  
  \If{$p>1$}{
    $r:={\tt NormExp}(x)$\;
    Let $x'=x\,S_r(x) - 2 S_{r-1}(x)$, $y'=y$, $z' = S_r(x)z - S_{r-1}(x)y$, and replace
    $(x,y,z)$ by $(x',y',z')$.\;
  } 
  \lIf{$\vert z\vert < y$}{Interchange $y$ and $z$.}
  \uIf{$y\ge 2$ and $z\ge 2$}{
  Replace $(x,y,z)$ by $(x,y,xy-z)$ and continue with step~2.\;}
  \lElseIf{$y\ge 2$ and $z\le -2$}{\Return $(x,y,z)$.}
  
  Apply Algorithm~\ref{alg:RatAngle} to~$y$ and get a pair $(p',q')$.\;
  \lIf{$(p',q')=(0,0)$}{\Return {\tt "not discrete"}.}  
  \If{$p'>1$}{
    $r':={\tt NormExp}(y)$\;    
    Let $x'=x$, $y' = y\,S_{r'}(y) - 2 S_{r'-1}(y)$, $z'= S_{r'}(y)z - S_{r'-1}(y)x$,
    and replace $(x,y,z)$ by $(x',y',z')$.\;
  }
 \lIf{$y<x$}{Interchange $x$ and $y$.}
 \uIf{$z\ge 2$}{
   Replace $(x,y,z)$ by $(x,y,xy-z)$ and continue with step~2.\;}
 \ElseIf{$\vert z\vert <2$}{
    Apply Algorithm~\ref{alg:RatAngle} to~$\vert z\vert$ and get a pair $(p'',q'')$.\;
    \lIf{$(p'',q'')=(0,0)$}{\Return {\tt "not discrete"}.}
    \lIf{$p''>1$ {\rm and} $\vert z\vert < y$}{Continue with step~2.}   
 }
 \Return $(x,y,z)$.\; 
  \caption{The Extended Trace Minimization Algorithm}
  \label{alg:ExtTraceMin}
\end{algorithm}

\begin{theorem}{\bf (Extended Trace Minimization Algorithm)}\label{thm:ExtTraceMin}\,\\
Let $A,B\in \PSLR$ be such that $\langle A,B\rangle$
contains elliptic elements and $\tr([A,B])>2$. Then Algorithm~\ref{alg:ExtTraceMin} 
returns {\tt "not discrete"} or computes
a trace triple $(x,y,z)$ of a pair $(U,V)\in \PSLR^2$ such that the following conditions
are satisfied.
\begin{enumerate}
\item[(a)] The pair $(U,V)$ is obtained from $(A,B)$ by a series of Nielsen transformations 
and extended Nielsen transformations.

\item[(b)] The resulting triple $(x,y,z)$ satisfies $0\le x\le y$.

\item[(c)] If $y\ge 2$ then $z\le -2$.

\item[(d)] The element~$x$ is of the form $x = 2 \cos(\pi/p)$ for some $p\ge 2$.

\item[(e)] If $y<2$ then $y$ is of the form $y = 2 \cos(\pi/q)$ for some $q\ge p$.

\item[(f)] If $y<2$ then $z\le -2$ or $\vert z\vert \ge y$ and 
$\vert z\vert = 2\,\cos( r\,\pi/s)$ with $s\ge 2$, with $1\le r\le s-1$, and with $\gcd(r,s)=1$. 

\item[(g)] $2 < \tr([U,V]) \le \tr([A,B])$

\end{enumerate}
\end{theorem}

\begin{proof}
First we show that the algorithm can always be executed and is finite. Since $\langle A,B\rangle$
is not free, the element $z=\tr(UV)$ computed in step~1 satisfies $-2< z<0$
by Theorem~\ref{thm:classify}.
Hence step~2 can be executed and yields a triple with $0\le x < 2$. If step~3 yields
$(\bar{p},\bar{q})\ne (0,0)$, the corresponding matrix has a finite order $\bar{q}\ge 2$ and~$x$ 
is of the form required to perform the
normalization in steps~5-8. Note that~$y$ is not changed here. 

If $y\ge 2$ and $z\ge 2$ in step~10, we perform a trace minimizing step and
continue with step~2. It follows as in the usual trace minimizing algorithm 
that the condition $z\ge 2$ can happen only finitely many times.

If $y\ge 2$ and $z\le -2$, the algorithm stops with the output $(x,y,z)$.
Thus we have $0\le y<2$ in step~13.
If step~13 yields $(p',q')\ne (0,0)$ then
the corresponding matrix has finite order, and we can perform 
the normalization of~$y$ in steps~15-18. Notice that~$x$ is not changed here.

If the algorithm outputs {\tt "not discrete"}, this is clearly correct, as the
elliptic element with this trace has infinite order.
In view of step~19, we have $0\le x\le y<2$ in step~20.

If $z\ge 2$ here, we perform a trace minimizing step and continue
with $(x,y,xy-z)$ in step~2. Again, it follows as in the usual trace minimizing
algorithm that this can happen only finitely many times.
If we have to return to step~2 from step~25, a normalization step has to
be applied to~$z$. This can happen only finitely many times because 
of Proposition~\ref{prop:NormStep}.d.

It remains to show that the resulting triple $(x,y,z)$ has the claimed properties.

Claim~(a) holds, because all changes to the tuple $(x,y,z)$ correspond to Nielsen equivalences or 
normalization steps. Claim~(b) holds if we stop in step~12, because at this point
$0\le x<2$ and $y\ge 2$. If we stop in step~26, claim~(b) is a consequence of
step~19.

To show~(c), we note that we must have stopped in step~12 and the claim follows.
Claim~(d) is a consequence of the normalization in steps~5-8, and of 
the one in steps~15-18 if~$x$ and~$y$ were interchanged in step~19.
To verify~(e), we note that~$y$ must have been output in step~27, and then the
claim follows from the normalizations of~$x$ and~$y$ in steps~5-8 and~15-18. 

Claim~(f) follows, since the loop condition guarantees $z<2$, and
Algorithm~\ref{alg:RatAngle} yields the claim if $-2<z<2$.

Finally, we prove~(g). 
By Proposition~\ref{prop:trcomminv}, the Nielsen reduction steps underlying
steps 2,9,11,19,21 do not change the trace of the commutator of the matrices corresponding to~$x$ and~$y$
by Proposition~\ref{prop:trcomminv}. Initially, this number is $\tr([A,B])>2$.  
The normalizations in steps~5-8 and (if applicable) in steps~15-18 decrease it, 
but keep it larger than~2 by Proposition~\ref{prop:NormStep}.d.
\end{proof}

\begin{remark}
As for the usual Trace Minimization Algorithm, it is possible to include the 
construction of words $u,v$ in the letters $A,B,A^{-1}, B^{-1}$ 
such that the resulting trace triple $(x,y,z)$ of Algorithm~\ref{alg:ExtTraceMin}
corresponds to the pair $(U,V)$ given by $U=u(A,B,A^{-1},B^{-1})$
and $V=v(A,B,A^{-1},B^{-1})$.

Moreover, we note that in an actual implementation of Algorithm~\ref{alg:ExtTraceMin},
one would certainly include the fractions corresponding to the computed rational angles
in the output.
\end{remark}

%
%

\section{Roots Generating Free Fuchsian Groups of Rank 2}
\label{sec4}

In this section we prove Theorem~\ref{thm:free}.
For this purpose, we first recall a classification of free Fuchsian groups of rank~2
from~\cite{Ros1}. The papers~\cite{Ros2,FR} provide an even more detailed classification of all
generating pairs of 2-generated Fuchsian groups. 

Recall that a discrete free subgroup of~$\PSLR$ contains no elliptic elements.
This follows from the facts that a free group contains no elements of finite order and
that the existence of elliptic elements of infinite order contradicts discreteness
(e.g., see~\cite[Thm.~8.4.1]{Bea1}).

\begin{theorem}[\textbf{Classification of 2-Generated Free Fuchsian Groups}]$\mathstrut$\label{thm:classify}
Let $A,B \in \PSLR$. Then $\langle A,B\rangle$ is a free Fuchsian group of rank~2
if and only if one of the following two cases occurs.
\begin{enumerate}
\item[(a)] $\tr([A,B]) \le -2$

\item[(b)] There exists a Nielsen transformation from $(A,B)$ to a pair $(U,V)$
of matrices $U,V\in \PSLR$ with $2 \le \tr(U) \le \tr(V)$ and $\tr(UV) \le -2$.

\end{enumerate}
Moreover, in case~(b) we have $\tr([A,B]) \ge 18$.
\end{theorem}

\begin{proof}
This follows from~\cite[Satz~1]{Ros1}. Notice that
a free product of two infinite cyclic groups is isomorphic to a free group
of rank~2. Furthermore, recall that an element~$A$ of infinite order in
a discrete subgroup of~$\PSLR$ satisfies $\vert \tr(A) \vert \ge 2$, since elliptic 
elements in such groups have finite order (e.g., see~\cite[Thm.~8.4.1]{Bea1}).
The fact that the pair $(U,V)$ in~(b)
is obtained from $(A,B)$ by a Nielsen transformation is shown in the proof 
of~\cite[Satz~1]{Ros1}. 

Finally, in case~(b) we have $\tr([A,B]) = \tr([U,V])$ by Proposition~\ref{prop:trcomminv} and
$$
\tr([U,V]) = \tr(U)^2 + \tr(V)^2 + \tr(UV)^2 
- \tr(U)\tr(V)\tr(UV) - 2 \ge 18.
$$
by Proposition~\ref{prop:traceprops}.d.
\end{proof}

In the following, let $A,B,R,S\in \PSLR$ be elements such that $R^m=A$ and $S^n=B$ for 
some $m,n\in\mathbb{N}_+$. Recall that we chose representatives such that $\tr(A), \tr(B),
\tr(R)$, and $\tr(S)$ are non-negative.
Based on the above classification, it is straightforward to prove Theorem~\ref{thm:free}
as follows.

\begin{proof}[{\bf Proof of Theorem~\ref{thm:free}}]
Since we assumed that $\langle A,B\rangle$ is a free Fuchsian group of rank~2,
Theorem~\ref{thm:classify} allows us to distinguish two separate cases 
in which we can check whether $\langle R,S\rangle$ is a free Fuchsian group of rank~2.

\medskip
\noindent{\bf Case F1:} $\tr([A,B])\le -2$.
By Proposition~\ref{prop:traceprops}.f, we have $\tr(A) >2$ and $\tr(B) >2$.
Then we can write $\tr(A)=2\,\cosh(\phi)$ with $\phi>0$, and hence
$\tr(R)= 2\,\cosh(\phi/m)>2$. Similarly, $\tr(S)>2$.
Proposition~\ref{prop:power}.c yields
\begin{align*}
\tr([A,B]) - 2 &\;=\; \tr([R^m,S^n]) - 2 &\\
&\;=\; S_m(\tr(R))^2 \, S_n(\tr(S))^2 \, (\tr([R,S]) -2), 
\end{align*}
and therefore 
$$
\tr([R,S]) = \frac{\tr([A,B]) -2}{S_m(\tr(R))^2 \, S_n(\tr(S))^2} + 2 < 2.
$$
Notice that $S_m(\tr(R))\ne 0$ and $S_n(\tr(S))\ne 0$ by Proposition~\ref{prop:root}.a.

Now Theorem~\ref{thm:classify} and $\tr([R,S])<2$ imply that $\langle R,S\rangle$
is a free Fuchsian group of rank~2 if and only if $\tr([R,S])\le -2$.
This is equivalent to 
$$
\frac{\tr([A,B]) -2}{S_m(\tr(R))^2 \, S_n(\tr(S))^2} \;\le\; -4
$$
and hence to $S_m(\tr(R))^2\, S_n(\tr(S))^2 \le \frac{1}{2} - \frac{1}{4}\, \tr([A,B])$, as claimed.

\medskip
\noindent{\bf Case F2:} As~$A$ and~$B$ have infinite order in a discrete subgroup of~$\PSLR$, 
we have $\tr(A) \ge 2$ and $\tr(B) \ge 2$. As above, it follows that $x=\tr(R)\ge 2$ and $y=\tr(S)\ge 2$.
By Proposition~\ref{prop:root}.a, it follows that $S_m(x)\ne 0$ and $S_n(y)\ne 0$.
By Proposition~\ref{prop:power}.c, we get $\tr([R,S]-2 = \frac{\tr([A,B]-2}{S_m(x)^2\, S_n(x)^2} > 0$. 

Hence we are in the case $\tau = \tr([R,S]) > 2$ of the Trace Minimalization
Algorithm\ref{alg:TraceMin}. It computes the trace triple $(x',y',z')$ of a pair $(U,V)$ of matrices 
in~$\PSLR$ which is Nielsen equivalent to~$(R,S)$
and satisfies $0\le x' \le y'$ as well as $z'<0$. 

Now, if the group $\langle R,S\rangle$ is a free Fuchsian group of rank~2 then
it contains no elliptic elements, and therefore $2\le x'\le y'$ as well as $z'\le -2$.

Conversely, if $2\le x'\le y'$ and $z'\le -2$ then $\langle U,V\rangle$ is a
free Fuchsian group of rank~2 by Theorem~\ref{thm:classify}.a. Thus the observation that
$\langle U,V\rangle = \langle R,S\rangle$ finishes the proof.
\end{proof}

A rather particular case occurs when we start
with two matrices $A,B \in \PSLR$ such that $\vert \tr(A)\vert = \vert \tr(B) \vert = 
\vert \tr(AB) \vert =2$. A similar case was studied in~\cite{Bea2}.
In our setting, we obtain the following result.

\begin{corollary}\label{cor:parabolic}
Let $A,B\in \PSLR$ be parabolic elements which generate a free Fuchsian group of rank~2.
Assume that their product $AB$ is parabolic, too.
Let $m,n\in\mathbb{N}_+$, and let $R,S\in \PSLR$ be such that $A=R^m$ and $B=S^n$.

Then the group $\langle R,S\rangle$ is a free Fuchsian group of rank~2 if and only if $m=n=1$.
\end{corollary}

\begin{proof}
As usual, we represent~$A$ and~$B$ by matrices in~$\SLR$ such that $\tr(A)=2$
and $\tr(B)=2$ are positive. Then $\tr(AB)=2$ is impossible, since this
would imply 
$$
\tr([A,B]) \;=\; \tr(A)^2 + \tr(B)^2 + \tr(AB)^2 - \tr(A)\tr(B)\tr(AB) - 2 = 2
$$
by Proposition~\ref{prop:traceprops}.d. Here $\langle A,B\rangle$ would be metabelian and
thus not a free Fuchsian group of rank 2.

Consequently, we have $\tr(AB)=-2$. By Theorem~\ref{thm:CompTraceTriple}.a, we have
$\tr(R) = \tr(S) =2$. Hence Proposition~\ref{prop:SnRules}.f yields $S_m(\tr(R))=m$ 
and $S_n(\tr(S)) =n$. Plugging these values into
the formula in Proposition~\ref{prop:root}.d yields $\tr(RS) = 2 - \frac{4}{mn} < 2$.
In the current situation, $\langle R,S\rangle$ is a free Fuchsian group of rank~2 if and only
if $\tr(RS) \le -2$. This is equivalent to $m=n=1$, as claimed.
\end{proof}

To conclude this section, we illustrate Theorem~\ref{thm:free} with a concrete example.

\begin{example}
Suppose we are given the trace triple $(194,2627658,61714)$. For instance, it corresponds to
the pair $(A,B)$ given by $A = \left( \begin{smallmatrix} -1 & 28 \sqrt{6} + 70 \\ 28 \sqrt{6}-70 & 195
\end{smallmatrix} \right)$ and $B= \left( \begin{smallmatrix} 2627796 & -19403 \\ 19403 & -138
\end{smallmatrix} \right)$ in $\PSLR$. Using Theorem~\ref{thm:classify}, 
we can check that $\langle A,B\rangle$ is a free Fuchsian group of rank~2.

Now we let $(R,S) \in \PSLR^2$ such that $R^2=A$ and $S^3=B$. We want to check 
whether $\langle R,S\rangle$ is a free Fuchsian group of rank~2. 

First we compute the trace triple of $(R,S)$ using Theorem~\ref{thm:CompTraceTriple}.
We get $(x,y,z)=(14, 138, 10)$. Then we calculate
$\tau=\tr([R,S]) = 18$ and conclude that we need to apply the Trace Minimizing Algorithm~\ref{alg:TraceMin}
in the case $\tau>2$. Let us follow the steps.
\begin{enumerate}
\item[(0)] $(x_0,y_0,z_0) = (14,138,10)$
\item[(1)] $(x_1,y_1,z_1) = (10,14,138)$
\item[(2)] $x_1 y_1 - z_1 = 2$ and $(x_2,y_2,z_2) = (2,10,14)$
\item[(3)] $x_2 y_2 - z_2 = 6$ and $(x_3,y_3,z_3) = (2,6,10)$
\item[(4)] $x_3 y_3 - z_3 = 2$ and $(x_4,y_4,z_4) = (2,2,6)$
\item[(5)] $x_4 y_4 -z_4 = -2$ and $(x_5,y_5,z_5) = (-2,2,2)$
\item[(6)] The algorithm returns $(x',y',z') = (2,2,-2)$.
\end{enumerate}
Altogether, we find that a pair $(R,S)$ with $R^2=A$ and $S^3=B$ generates a free 
Fuchsian group of rank~2.

For instance, for the given elements $A,B$, this pair is given by $R= \left(\begin{smallmatrix}
0 & 2\sqrt{6}+5\\ 2\sqrt{6}-5 & 14 \end{smallmatrix}\right)$ and $S= \left( \begin{smallmatrix}
138 & -1\\ 1 & 0 \end{smallmatrix}\right)$.
\end{example}

%
%

\section{Roots Generating Fuchsian Groups with Torsion}
\label{sec5}

In this section we treat the case of roots of 2-generated Fuchsian groups
which generate Fuchsian groups with torsion.

More precisely, we let $A,B,R,S \in \PSLR$ such that $\langle A,B\rangle$
is a Fuchsian group and $R^m=A$, $S^n=B$ for some $m,n\in\mathbb{N}_+$.
Recall that $\tr([A,B]\ne 2$, as $\langle A,B\rangle$ is non-elementary.
When we ask the question whether $\langle R,S\rangle$ is a discrete group,
we may distinguish the following cases:

\smallskip
\noindent{\bf Case I:} $-2 < \tr([A,B]) < 2$, i.e., the commutator $[A,B]$ is an
elliptic element. In this case, Proposition~\ref{prop:traceprops}.f implies $\tr(A)>2$
and $\tr(B)>2$.

\smallskip
\noindent{\bf Case II:} $\tr([A,B])>2$, $\tr(A)\ge 2$, $\tr(B)\ge 2$,
and $\langle A,B\rangle$ is not free.

\smallskip
\noindent{\bf Case III:} $\tr([A,B])>2$, $0\le \tr(A)<2$.

\smallskip
\noindent{\bf Case IV:} $\langle A,B\rangle$ is free.
\medskip

Notice that $\langle A,B\rangle$ is not free in the first three cases.
The classification of the possible presentations (or signatures) of 2-ge\-ner\-a\-ted Fuchsian
groups is a classical result (see~\cite{Kna}, \cite{Pur2}, \cite{Ros2}). For the proofs of 
our theorems, we need detailed versions (as given in~\cite{Kna}, \cite{Ros2}, \cite{FR}) which we
recall for the convenience of the readers.

\bigskip
\noindent{\bf Case I (Elliptic Commutator):} Assume that $-2 < \tr([A,B]) <2$. For this case, 
the following classification result was shown in~\cite[Theorem 3]{Ros2}.

\begin{theorem}[2-Generated Fuchsian Groups with Elliptic Commutator]\label{thm:ClassifyElliptic}\, \\
If $-2 < \tr([A,B]) <2$, the group $G=\langle A,B\rangle$ is discrete if and only if one of the
following cases occurs.
\begin{enumerate}
\item[(a)] $\tr([A,B]) = -2\cos(\pi/p)$ for some $p\ge 2$.

\item[(b)] $\tr([A,B]) = -2\cos(2\pi/p)$ for some odd number $p\ge 3$.

\item[(c)] $\tr([A,B]) = -2\cos(6\pi/r)$ for some $r\ge 7$ such that $\gcd(r,6)=1$,
and the Trace Minimization Algorithm~\ref{alg:TraceMin}, applied to the trace triple
of $(A,B)$, yields a trace triple $(\tr(U),\tr(V),\tr(UV))$ such that $\tr(U)=\tr(V)=\tr(UV)$.

\item[(d)] $\tr([A,B]) = -2\cos(4\pi/r)$ for some odd number $r\ge 5$,
and the Trace Minimization Algorithm~\ref{alg:TraceMin}, applied to the trace triple
of $(A,B)$, yields a trace triple $(\tr(U),\tr(V),\tr(UV))$ such that $\tr(U)=\tr(V)$
and such that $\tr(UV)= \frac{1}{2} \tr(U)^2$.

\item[(e)] $\tr([A,B]) = -2\cos(3\pi/r)$ for some $r\ge 4$ with $\gcd(r,3)=1$,
and the Trace Minimization Algorithm~\ref{alg:TraceMin}, applied to the trace triple
of $(A,B)$, yields a trace triple $(\tr(U),\tr(V),\tr(UV))$ such that $\tr(U)=\tr(V)=\tr(UV)$.

\item[(f)] $\tr([A,B]) = -2\cos(4\pi/7)$,
and the Trace Minimization Algorithm~\ref{alg:TraceMin}, applied to the trace triple
of $(A,B)$, yields a trace triple $(\tr(U),\tr(V),\allowbreak \tr(UV))$ such that $\tr(U)+1=\tr(V)=\tr(UV)$.

\end{enumerate}
\end{theorem}

In this theorem, case~(a) yields a one-cone torus group. Case~(b) yields
the amalgamated free product of $H_1 = \langle s_1,s_2 \mid (s_1)^2 = (s_2)^2\rangle$ 
and $H_2 = \langle s_3, s_4 \mid (s_3)^2 = (s_4)^p\rangle$ with cyclic 
amalgamated subgroup $A = \langle s_1 s_2\rangle = \langle s_3 s_4\rangle$.
Case~(c) yields a $(2,3,r)$-triangle group. Case~(d) yields
a $(2,4,r)$-triangle group. Case~(e) yields a $(3,3,r)$-triangle group. Finally, cases~(e) and~(f) yield
$(2,3,7)$-triangle groups. A check of the various cases shows that the following bound holds. 

\begin{corollary}\label{cor:TrLess1}
For a 2-generated Fuchsian group with $-2< \tr([A,B]) < 2$, we have $-2 < \tr([A,B]) \le 1$.
\end{corollary}

Next we write $\tr(A)=2\cosh(\phi_1)$ and $\tr(B)=2\cosh(\phi_2)$ with $\phi_1,\phi_2\in\mathbb{R}$.
Given $m,n\in\mathbb{N}_+$ and elements $R,S \in \PSLR$ such that $R^m=A$ and $S^n=B$,
we have $\tr(R)=2\cosh(\phi_1/m)$ and $\tr(S)=2\cosh(\phi_2/n)$. Thus~$R$ and~$S$ are hyperbolic elements
and the group $\langle R,S\rangle$ contains elliptic elements, because $\langle A,B\rangle$ does.
In this setting we have the following theorem.

\begin{theorem}[{\bf Case I}]\label{thm:EllipticComm}\, \\
Let $A,B\in\PSLR$ such that $\langle A,B \rangle$ is a Fuchsian group. Assume that 
$-2 < \tr([A,B]) < 2$ and that $R,S\in \PSLR$ satisfy $R^m=A$, $S^n=B$ for some $m,n\ge 1$.
Then the group $\langle R,S\rangle$ is discrete if and only if $m=n=1$.
\end{theorem}

\begin{proof} By Proposition~\ref{prop:traceprops}.f, we know that $A,B$ are hyperbolic elements.
Hence also $R,S$ are hyperbolic and Proposition~\ref{prop:power}.c yields
$$
\tr([R,S]) \;=\; 2 \;+\; \frac{\tr([A,B]) - 2}{S_m(\tr(R))^2\, S_n(\tr(S))^2} < 2
$$
Since $\langle R,S\rangle$ is not free, we also have $\tr([R,S]) > -2$ by Theorem~\ref{thm:classify}.

Consequently, we can apply Theorem~\ref{thm:ClassifyElliptic} to the group $\langle R,S\rangle$,
and Corollary~\ref{cor:TrLess1} shows $\tr([R,S]) \le 1$. Thus the above formula for $\tr([R,S])$
yields
$$
S_m(\tr(R))^2\, S_n(\tr(S))^2 \le 2 - \tr([A,B]) < 4.
$$
By Proposition~\ref{prop:MoreSnProps}.d, we have the inequalities $S_m(\tr(R))^2 > m^2 \ge 4$ for $m\ge 2$ and
$S_n(\tr(S))^2 > n^2 \ge 4$ for $n\ge 2$. Therefore we conclude that $m=n=1$.
\end{proof}

\bigskip
\noindent{\bf Case II (Hyperbolic Commutator, Elliptic Elements, Non-Elliptic Generators):}\\
For 2-generated non-elementary subgroups of $\PSLR$ with hyperbolic commutator and 
containing elliptic elements, we use the following classification theorem (cf.~\cite[Theorem~2]{Ros2}).

\begin{theorem}[2-Generated Fuchsian Groups with Elliptic Elements]\label{thm:TheoremC}\,\\
Let $A,B\in \PSLR$ be non-elliptic elements such that $\tr([A,B])>2$
and such that $\langle A,B\rangle$ is non-elementary and contains elliptic elements.
Then $\langle A,B\rangle$ is discrete if and only if there exists an extended Nielsen transformation
from $(A,B)$ to a pair $(U,V)$ such that (after a suitable choice of signs) the following 
conditions hold.
\begin{enumerate}
\item[(1)] $\tr(U)=2 \cos(\pi/p)$ for some $p\ge 2$.

\item[(2)] $\tr(V)=2 \cos(\pi/q)$ for some $q\ge 2$ or $\tr(V)\ge 2$.

\item[(3)] $\tr(UV)= - 2 \cos(\pi/r)$ for some $r\ge 2$ or $\tr(UV)\le -2$.
\end{enumerate}
\end{theorem}

For the case in which the Extended Trace Minimization Algorithm~\ref{alg:ExtTraceMin} returns
a trace triple $(x,y,z)$ with $0\le x\le y<2$, we make use of the
following classification theorem by A.W.\ Knapp (cf.~\cite[Theorem~2.3]{Kna}).

\begin{theorem}[2-Generated Fuchsian Groups with Elliptic Generators]\label{thm:Knapp}\, \\
Let $A,B\in \PSLR$ be elliptic elements such that $\tr(A) = 2\,\cos(\pi/p)$
with $p\ge 2$ and $\tr(B)=2\,\cos(\pi/q)$ with $q\ge p$. Then $\langle A,B\rangle$ is discrete
if and only if one of the following cases occurs.
\begin{enumerate}
\item[(a)] $\tr(AB)\le -2$

\item[(b)] $\tr(AB) = -2\, \cos(r\,\pi/s)$ for some $s\ge 2$ and $r\in \{1,\dots,s-1\}$
with $\gcd(r,s)=1$, such that one of the following conditions holds:
\begin{itemize}
\item[(i)] $r=1$

\item[(ii)] $p=q$, $r=2$, and $s\ge 3$ odd.

\item[(iii)] $p=2$, $q\ge 3$ odd, $r=2$, $s=q$.

\item[(iv)] $p=3$, $q\ge 7$, $\gcd(q,3)=1$, $r=3$, $s=q$.

\item[(v)] $p=q=s$, $r=4$, $s\ge 7$ odd.

\item[(vi)] $p=3$, $q=7$, $r=2$, $s=7$
\end{itemize}

\end{enumerate}
\end{theorem}

With these classifications in mind, we can now prove Case~II in Theorem~\ref{thm:torsion}
and Case III in Theorem~\ref{thm:TorsionElliptic}. We start by computing the trace triple $(\tr(R),\tr(S),\tr(RS))$
using Theorem~\ref{thm:CompTraceTriple}. 

\begin{proof}[{\bf Proof of Case II}]
First we note that $x=\tr(R)\ge 2$ and $y=\tr(S)\ge 2$ since~$A$ and~$B$ are parabolic or hyperbolic.
Therefore $S_m(x)\ne 0$, $S_n(y)\ne 0$, and
$$
\tr([R,S]) \;=\; 2 \;+\; \tfrac{\tr([A,B]) -2}{S_m(x)^2\, S_n(y)^2} > 2.
$$
Consequently, we can apply the Extended Trace Minimization Algorithm~\ref{alg:ExtTraceMin} to~$(x,y,z)$
and get $(x',y',z')$ with the properties listed in Theorem~\ref{thm:ExtTraceMin}.

First of all, we know that $x'=2\,\cos(\pi/p)$ for some $p\ge 2$.
Now, if $y'\ge 2$ then $z'\le -2$ and $\langle R,S\rangle$ is discrete by Theorem~\ref{thm:TheoremC}.
Further, if $0\le y'< 2$, we know that $y'=2\,\cos(\pi/q)$ with $q\ge p$.
In this case, if $z'\le -2$ then  $\langle R,S\rangle$ is again discrete by Theorem~\ref{thm:TheoremC}
or Theorem~\ref{thm:Knapp}.a.

Finally, we are left with the case $x'= 2\,\cos(\pi/p)$, where $p\ge 2$,
$y'= 2\,\cos(\pi/q)$, where $q\ge p$, and $z'= -2\,\cos(r\,\pi/s)$, where
$s\ge 2$, $1\le r\le s-1$, and $\gcd(r,s)=1$. In this case $\langle R,S\rangle$
is discrete if and only if one of the conditions (i) - (vi) in Theorem~\ref{thm:Knapp}.b holds.
Altogether, the proof of Case II is complete.
\end{proof}

\begin{proof}[{\bf Proof of Case III}]
To begin with, we claim that $\tr([R,S])>2$. If we had $-2<\tr([R,S])<2$, then
$\tr(R)>2$ and $\tr(S)>2$ by Proposition~\ref{prop:traceprops}.f, in contradiction
to $R^m=A$ and $\ord(A)<\infty$, as~$A$ is elliptic in a Fuchsian group $\langle A,B\rangle$.
Moreover, if $\tr([R,S])\le -2$ then $\langle R,S\rangle$ would be free by Theorem~\ref{thm:classify},
in contradiction to $\ord(A)<\infty$.

So, we have $\tr([R,S])>2$ and $\langle R,S\rangle$ contains an element of finite order.
Hence we can apply the Extended Trace Minimization Algorithm~\ref{alg:ExtTraceMin} to the trace triple of $(R,S)$.
By the discussion in the proof of Case~II, its output satisfies the stated conditions
if and only if $\langle R,S\rangle$ is discrete.
\end{proof}

The result of this discreteness check may depend on the choice of a number $\ell\in \{0,\dots,m-1\}$,
i.e., on the choice of the element~$R$ with $R^m=A$, respectively the choice of~$\tr(R)$,
as the following examples shows.

\begin{example}\label{ex:depends}
Consider the trace triple $(0,2,1)$ which corresponds for instance to the pair
$(A,B) \in \PSLR^2$ given by $A= \left( \begin{smallmatrix} 0 & 1\\ -1 & 0\end{smallmatrix}\right)$
and $B= \left( \begin{smallmatrix} 1 & 0 \\ 1 & 1 \end{smallmatrix} \right)$. We choose $m=3$ and $n=1$,
so $S=B$ and $\tr(S)=2$. For $\tr(R)$, there are three possible values according to 
Theorem~\ref{thm:CompTraceTriple}.c, depending on the choice of $\ell\in \{0,1,2\}$.

\smallskip
\noindent{\bf Case $\ell=0$:} Here the matrix~$R$ satisfies $\tr(R) = 2\, \cos(\pi/6) = \sqrt{3}$,
i.e., $R$ is a normalized elliptic element of order~6.
We calculate $\tr(RS) = \frac{1}{2} + \sqrt{3} \approx 2.232$ and $\tr([R,S]) =  \frac{9}{4} = 2.25$.
Thus we are in the setting of the Extended Trace Minimization Algorithm~\ref{alg:ExtTraceMin}.
Let us follow its steps.
\begin{align*}
(\sqrt{3},\;2,\; &\tfrac{1}{2} + \sqrt{3}) \;\longrightarrow\; (\sqrt{3},\; 2,\; -\tfrac{1}{2} + \sqrt{3})
\;\longrightarrow\; ( - \tfrac{1}{2} + \sqrt{3}, \; \sqrt{3},\; 2)\\ 
&\;\longrightarrow\;
( - \tfrac{1}{2} + \sqrt{3},\; \sqrt{3},\; 1 - \tfrac{1}{2}\,\sqrt{3}) \;\longrightarrow \;
( 1 - \tfrac{1}{2}\,\sqrt{3},\;  - \tfrac{1}{2} + \sqrt{3},\; \sqrt{3})\\
&\;\longrightarrow\;  (1 - \tfrac{1}{2}\,\sqrt{3},\;  - \tfrac{1}{2} + \sqrt{3},\; -2 + \tfrac{1}{4}\, \sqrt{3})
\end{align*}

Next we have to apply the Rational Angle Recovery Algorithm~\ref{alg:RatAngle} to
$x' = 1 - \frac{1}{2}\,\sqrt{3}$. The minimal polynomial of~$x'$ is $P(t)=t^2-2t + \frac{1}{4}$.
The numbers~$q_i$ such that $\Phi(2q_i) = 2\,\deg(P(t)) = 4$ are $q_i\in \{5,8,10,12\}$, 
but $P(t)$ does not divide $S_{q_i}(t)$ for any of them. Hence the algorithms returns $(0,0)$,
and the group $\langle R,S\rangle$ is not discrete.

\smallskip
\noindent{\bf Case $\ell=1$:} In this case, the matrix $\widetilde{R}$ satisfies 
$\tr(\widetilde{R})=2\, \cos (3\,\pi/6) = 0$. So, $\widetilde{R}$ is elliptic of order~2,
and $\tr(\widetilde{R}S)= 1$ implies that $\widetilde{R}S$ is elliptic of order~3.
Here we get $\tr([\widetilde{R},S]) = 3$ and can apply the Extended Trace Minimization Algorithm~\ref{alg:ExtTraceMin}
again. We start with the usual Trace Minimizing Algorithm~\ref{alg:TraceMin} and compute
$$
(0,2,1) \;\longrightarrow\; (0,1,2) \;\longrightarrow\; (0,1,-2)
$$
Here the first two traces are already normalized, and the third one implies that
$\langle \widetilde{R},S\rangle$ is discrete by Case~II.b. In fact, this group is conjugated to the
modular group. 

Altogether, we find that $\ell=0$ leads to a non-discrete group and $\ell=1$ leads to a discrete group.
\end{example}

It remains to consider the case when $\langle A,B\rangle$ is a free Fuchsian group.
We have already characterized the groups $\langle R,S\rangle$ which are free Fuchsian groups
of rank~2 in Theorem~\ref{thm:free}. Thus all is left is to characterize when $\langle R,S\rangle$
is discrete and has torsion elements.

\begin{proof}[{\bf Proof of Case IV}]
Of course, if the result of the Trace Minimization Algorithm~\ref{alg:TraceMin} tells us via Theorem~\ref{thm:free}
that $\langle R,S\rangle$ is a free Fuchsian group of rank~2, we are done. So, let us now assume
that $\langle R,S\rangle$ contains elliptic elements. We compute the trace triple of $(R,S)$ and
$\tau=\tr([R,S])$. If $-2<\tau<2$ then we can apply the classification of 2-generated Fuchsian groups
with elliptic commutator (see Theorem~\ref{thm:EllipticComm}) to the group $\langle R,S\rangle$
and get part~(b) of Theorem~\ref{thm:CaseIV}.

The remaining possibility is $\tau>2$. As the group $\langle R,S\rangle$ contains an elliptic element, 
it can be discrete only if this elliptic element has a finite order. Hence we are in a position to
apply the Extended Trace Minimization Algorithm~\ref{alg:ExtTraceMin}, and its result allows us
to check whether $\langle R,S\rangle$ is discrete in the same way as in Case~II.
\end{proof}

%
%

\section{The Rational Power Algorithm}
\label{sec6}

As mentioned before, J.~Gilman's paper~\cite{G3} contains several geometric conditions 
for the result of adjoining roots to a Fuchsian group, chiefly in
the cases of Theorem~\ref{thm:free}.
In addition, she addressed also the following slightly more general question.

Let $A,B \in \PSLR$ be elements which generate a Fuchsian group $\langle A,B\rangle$. 
Let $\frac{p}{q},\, \frac{p'}{q'}\in \mathbb{Q}_+$ with $p,q,p',q' \in \mathbb{N}_+$
and $\gcd(p,q)=\gcd(p',q')=1$. Let $R,S \in \PSLR$ be elements such that
$A=R^{\,p/q}$ and $B=S^{\,p'/q'}$. Under which conditions is the group $\langle R,S\rangle$
again a Fuchsian group? If $\langle A,B \rangle$ is a free Fuchsian group of rank~2,
under which conditions is $\langle R,S\rangle$ again a free Fuchsian group of ranke~2?

These questions can be answered in a straightforward way using the methods developed here.
Clearly, the elements $\widetilde{A}=A^q$ and~$\widetilde{B}=B^{\, q'}$
generate a Fuchsian group, and if $\langle A,B\rangle$ is free of rank~2, then $\langle A^q, B^{\,q'}\rangle$
is free of rank~2, as well. The trace triple $(\tr(A^q), \tr(B^{\,q'}), \tr(A^q B^{\,q'}))$ can be calculated
using Proposition~\ref{prop:power}, and $\tr([A^q,B^{\,q'}])$ is then obtained from
Proposition~\ref{prop:traceprops}.d.
Now we can apply Theorems~\ref{thm:free}, \ref{thm:torsion}, \ref{thm:TorsionElliptic},
and~\ref{thm:CaseIV} to $R^p=\widetilde{A}$ and $S^{p'}=\widetilde{B}$.

\begin{theorem}\label{thm:RatPowerFree}
Let $A,B \in \PSLR$ be elements which generate a free Fuchsian group of rank~2, 
and let $\frac{p}{q},\frac{p'}{q'} \in \mathbb{Q}_+$ with $p,q,p',q' \in \mathbb{N}_+$ and 
$\gcd(p,q)=\gcd(p',q')=1$. Let $R,S\in \PSLR$ such that $R^{\, p/q}=A$ and $S^{\, p'/q'}=B$.
\begin{enumerate}
\item[(1)] Compute the trace triple of $\widetilde{A}=A^q$ and $\widetilde{B}=B^{\,q'}$ 
using Proposition~\ref{prop:power}, and calculate $\tau = \tr([\widetilde{A},\widetilde{B}])$.

\item[(2)] Compute the trace triple $(x,y,z) = (\tr(R),\tr(S),\tr(RS))$ of~$(R,S)$
using $R^{\, p} = \widetilde{A}$, $S^{\,p'}=\widetilde{B}$, and Theorem~\ref{thm:CompTraceTriple}.
\end{enumerate}
Then $\langle R,S\rangle$ is a free Fuchsian group of rank~2 if and only if one of the 
following two conditions hold:

\smallskip
\noindent{\rm\bf Case F1:} $\tau\le -2$ and 
$S_m(x)^2 \cdot S_n(y)^2 \;\le\; \tfrac{1}{2} - \tfrac{1}{4}\, \tau$.

\smallskip
\noindent{\rm\bf Case F2:} $\tau > 2$  and the trace triple $(x',y',z')$
which results from applying Algorithm~\ref{alg:TraceMin} to $(x,y,z)$ 
satisfies $x'\ge 2$, $y'\ge 2$, and $z'\le -2$.
\end{theorem}

\begin{theorem}\label{thm:RatPower}
Let $A,B \in \PSLR$ be elements which generate a Fuchsian group, 
and let $\frac{p}{q},\, \frac{p'}{q'}\in\mathbb{Q}_+$  with $p,q,p',q' \in \mathbb{N}_+$ 
and $\gcd(p,q)=\gcd(p',q')=1$. Perform steps~(1) and~(2) of Theorem~\ref{thm:RatPowerFree}.

Then $\langle R,S\rangle$ is a Fuchsian group if and only if one of the following
cases occurs:

\smallskip
\noindent{\rm\bf Case I:} $-2<\tau<2$ and $p=p'=1$.

\smallskip
\noindent{\rm\bf Case II:} $\tau>2$, $\tr(A^q)>2$, $\tr(B^{\,q'})>2$, $\langle R,S\rangle$
is not free, and the triple $(x',y',z')$ obtained by applying Algorithm~\ref{alg:ExtTraceMin}
to $(x,y,z)$ satisfies one of the conditions in Case~II of Theorem~\ref{thm:torsion}.

\smallskip
\noindent{\rm\bf Case III:} $A^q$ (and possibly $B^{\,q'}$) is elliptic,
and the triple $(x',y',z')$ obtained by applying Algorithm~\ref{alg:ExtTraceMin}
to $(x,y,z)$ satisfies one of the conditions in Case~II of Theorem~\ref{thm:torsion}.

\smallskip
\noindent{\rm\bf Case IV.a:} $\langle R,S\rangle$ is a free Fuchsian group of rank~2
according to Theorem~\ref{thm:RatPowerFree}.

\smallskip
\noindent{\rm\bf Case IV.b:} $\tau'= \tr([R,S])$ satisfies $-2 < \tau' < 2$, and Algorithm~\ref{alg:RatAngle},
applied to $\vert \tau'\vert$, yields a pair $(\tilde{p},\tilde{q})\ne (0,0)$ which satisfies one of the six 
conditions in Theorem~\ref{thm:CaseIV}.b.

\smallskip
\noindent{\rm\bf Case IV.c:} $\tau'>2$, $\langle R,S\rangle$ is not free, and Algorithm~\ref{alg:ExtTraceMin},
applied to $(x,y,z)$, yields a triple $(x',y',z')$ which satisfies one of the conditions in Case~II
of Theorem~\ref{thm:torsion}.
\end{theorem}

The proof of these theorems follows immediately by combining the calculation of the trace triple
of $(A^q,B^{\,q'})$ with the theorems in the introduction.

\bigskip
\noindent{\bf Acknowledgements.} The authors are very grateful to the referees of the original version
of this paper for their helpful comments which induced a major overhaul and many substantial improvements.

%
%

\end{document}